\documentclass[11pt,oneside]{amsart}
\usepackage{amssymb, amsmath, amsthm}

\usepackage{graphicx}
\usepackage{color}
\usepackage{enumerate}
\usepackage[colorlinks=true, linkcolor=blue, citecolor=blue]{hyperref}
\usepackage{amsrefs}
\usepackage{fullpage}
\usepackage{pinlabel}

\theoremstyle{plain}
\newtheorem{theorem}{Theorem}[section]
\newtheorem*{maintheorem}{Theorem \ref{Main Thm}}
\newtheorem*{krebestheorem}{Theorem \ref{Krebes thm}}

\newtheorem*{theorem*}{Theorem}

\newtheorem{proposition}[theorem]{Proposition}
\newtheorem{corollary}[theorem]{Corollary}
\newtheorem{lemma}[theorem]{Lemma}

\theoremstyle{definition}
\newtheorem{remarkx}[theorem]{Remark}
 \newenvironment{remark}
  {\pushQED{\qed}\remarkx}
  {\popQED\endremarkx}

\newtheorem{definition}[theorem]{Definition}

\theoremstyle{definition}

\newcommand{\R}{{\mathbb R}}

\newcommand{\Z}{\mathbb Z}

\newcommand{\Q}{\mathbb Q}

\newcommand{\nil}{\varnothing}

\newcommand{\defn}[1]{\emph{#1}}
\newcommand{\boundary}{\partial}

\newcommand{\nbhd}{\operatorname{Nb}} 
\newcommand{\interior}{\operatorname{int}} 

\renewcommand{\phi}{\varphi}

\newcommand{\co}{\mskip0.5mu\colon\thinspace}
\renewcommand{\circ}{\mathring}

\begin{document}


   \title[]{Closures of 1-tangles and annulus twists}
   \author{Scott A. Taylor}

  \begin{abstract}
  A 1-tangle is a properly embedded arc $\psi$ in an unknotted solid torus $V$ in $S^3$. Attaching an arc $\phi$ in the complementary solid torus $W$ to its endpoints creates a knot $K(\phi)$ called the closure of $\psi$. We show that for a given nontrivial 1-tangle $\psi$ there exist at most two closures that are the unknot. We give a general method for producing nontrivial 1-tangles admitting two distinct closures and show that our construction accounts for all such examples. As an application, we show that if we twist an unknot $q \neq 0$ times around an unknotted sufficiently incompressible annulus intersecting it exactly once, then there is at most one $q$ such that the resulting knot is unknotted and, if there is such, then $q = \pm 1$.  With additional work, we also show that the Krebes 1-tangle does not admit an unknot closure. Our key tools are the ``wrapping index'' which compares how two complementary 1-tangles $\phi_1$ and $\phi_2$ wrap around $W$, a theorem of the author's from sutured manifold theory, and theorems of Gabai and Scharlemann concerning band sums. 
    \end{abstract}

\maketitle


\section{Introduction}

Suppose that $V \cup_S W$ is a genus 1 Heegaard splitting of $S^3$ and that $\psi \subset V$ is a properly embedded arc. Given a properly embedded arc $\phi \subset W$ with $\boundary \phi = \boundary \psi$, we call $K(\phi)$ a \defn{closure} of $(V, \psi)$. We address the long-standing challenge of determining when $K(\phi)$ is or is not the unknot. We define a ``wrapping index'' $\omega(\phi_1, \phi_2)$ between two such arcs in $W$ (where $\phi_1$ is trivial) and prove:

\begin{theorem}\label{Main Thm}
If $(V, \psi)$ is nontrivial, then for any arc $\phi \subset W$ with $K(\phi)$ the unknot, the arc $\phi \subset W$ is trivial and, up to isotopy in $W$ relative to $\boundary \psi$, there are at most two such arcs $\phi$. Furthermore, if $\phi_1, \phi_2$ are two such arcs then $\omega(\phi_1, \phi_2) = 1$. 
\end{theorem}

Showing that $(W,\phi)$ is trivial when $K(\phi)$ is the unknot and $(V,\psi)$ is nontrivial is straightforward (Lemma \ref{trivial}); the significance of the theorem lies in the claim that there are at most two such arcs.  Figure \ref{Fig:Samples} gives several examples of unknot closures on nontrivial pairs $(V, \phi)$. Figure \ref{Fig:Samples}.C and \ref{Fig:Samples}.D show an example of a nontrivial $(V,\psi)$ having two distinct unknot closures. 

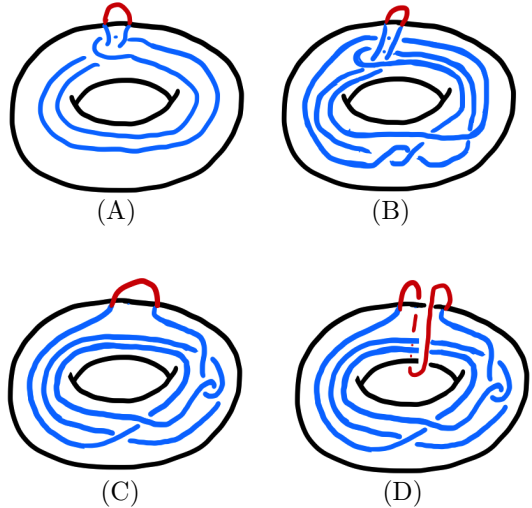
\begin{figure}[ht!]
\labellist
\small\hair 2pt
\pinlabel {(A)} [t] at 94 244
\pinlabel {(B)} [t] at 301 244
\pinlabel {(C)} [t] at 97 32
\pinlabel {(D)} [t] at 310 32
\endlabellist
\centering
\includegraphics[scale=0.5]{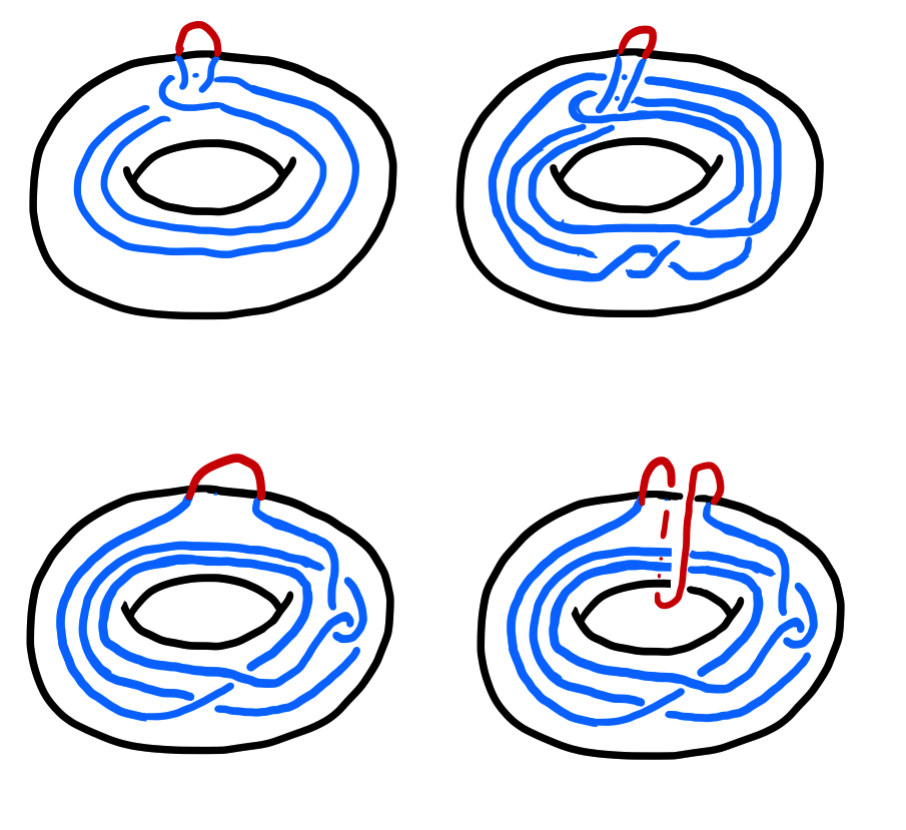}
\caption{Four examples of unknot closures of nontrivial 1-tangles. In each example, the solid torus $V$ is the inside of the black torus and the solid torus $W$ is the outside. The arc $\psi$ is the intersection of the knot with $V$ and the arc $\phi$ is the intersection of the knot with $W$. Note that in each case, the arc $\phi$ is isotopic in $W$, relative to its endpoints, into the torus. Also note that the pair $(V,\psi)$ is the same in both examples (C) and (D) but that the arcs $\phi$ are different.}
\label{Fig:Samples}
\end{figure}

In Section \ref{example}, we give a construction which produces many examples of $(V, \psi)$ where there are two such arcs $\phi$ such that $K(\phi)$ is the unknot. We also show that our construction accounts for all such examples.

On the other hand, a 1-tangle is \defn{persistent} if it does not admit an unknot closure.  Krebes raised the question \cite{Krebes}*{Section 14} of whether or not a specific 1-tangle, the ``Krebes 1-tangle'', shown in Figure \ref{Fig:Krebes} is persistent. Using a variation of the techniques used in the proof of Theorem \ref{Main Thm}, we show the following:

\begin{theorem}\label{Krebes thm}
The Krebes 1-tangle is persistent.
\end{theorem}

\begin{figure}[ht!]
\centering
\includegraphics[scale=0.5]{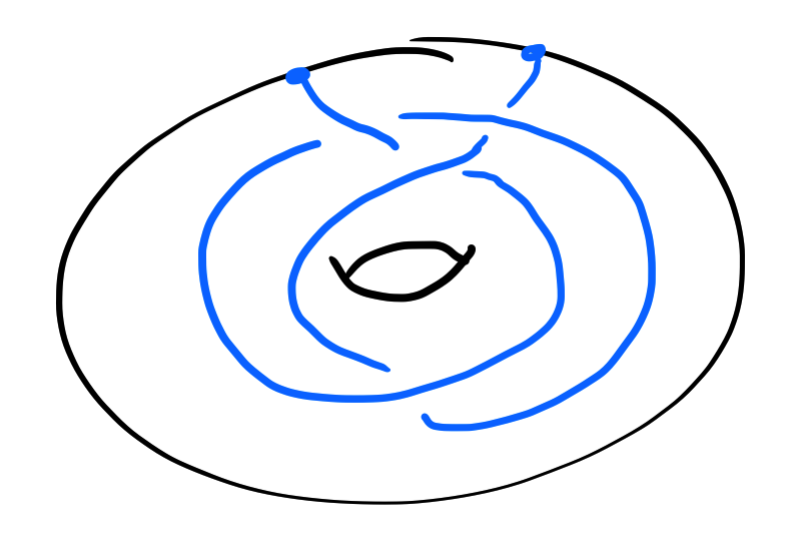}
\caption{The Krebes 1-tangle}
\label{Fig:Krebes}
\end{figure}

Previously, substantial, but partial, progress towards showing the persistence of the Krebes 1-tangle was made by Abernathy \cite{Abernathy} using branched double cover techniques, based on those in \cite{Ruberman}. She classified closures as ``even'' or ``odd'' and proved that if there is an unknot closure it must be even. In Section \ref{krebes section}, we observe that there is a very simple 3-coloring proof of that result. Abernathy and Gilmer \cite{AG} used skein techniques to show that 1-tangles to similar to Krebes' 1-tangle do not have unknot closures. Schreiber \cite{Schreiber} randomly generated closures of the Krebes 1-tangle and checked if they were the unknot. 

We also apply Theorem \ref{Main Thm} to the study of annulus twists on knots. Definition \ref{def: annulus twist} below defines what it means to twist a knot $K$ around an annulus $A$. Twisting $K$ $q$-times around $A$ produces a knot $K_q$; the sign of $q$ indicates the direction of the twist. Annulus twists are an effective technique for creating knots with certain types of properties (for example, \cites{AJLO, AT, DMM, BGL}) and substantial work has been done, under various hypotheses on $A$, on the asymptotics of certain invariants of $K_q$ (the knot obtained by twisting $K$ $q$ times using $A$) as $q \to \infty$. However, there seems to be no published work on the case when $A$ is unknotted. 

\begin{theorem}\label{Annulus Twist Theorem}
Suppose that $K_0 \subset S^3$ is an unknot and that $A \subset S^3$ is an unknotted sufficiently incompressible twisting annulus (Definition \ref{def: annulus twist}) for $K_0$ such that $|A \cap K_0| = 1$. Then there is at most one $q\neq 0$ such that $K_q$ is the unknot and, if $K_q$ is the unknot, then $q = \pm 1$.
\end{theorem}

Figure \ref{Fig:Annulus Twist} shows an example of an annulus twist on an unknot that produces another unknot. Note the connection to the 1-tangle closures in Figures \ref{Fig:Samples}.C and \ref{Fig:Samples}.D 

\begin{figure}[ht!]
\centering
\includegraphics[scale=0.5]{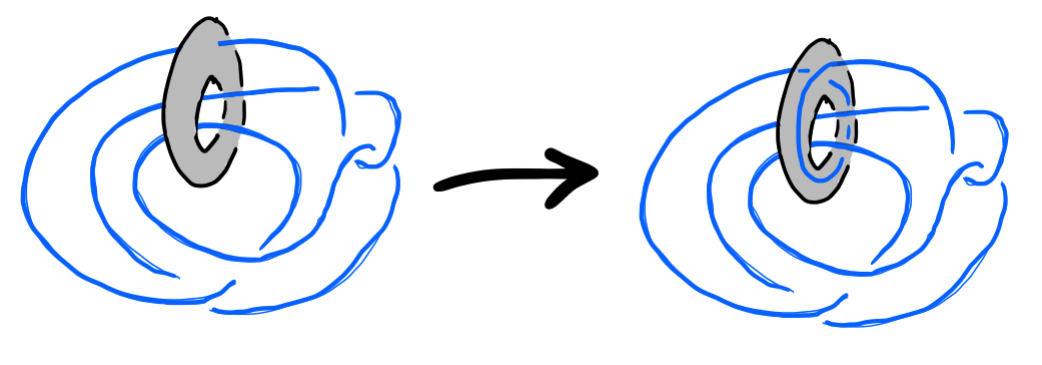}
\caption{An example where $K_0$ and one of $K_{\pm 1}$ are both unknots; it is another version of Figure \ref{Fig:Samples}.C and \ref{Fig:Samples}.D. By Theorem \ref{Annulus Twist Theorem}, if we had twisted more than once or in the opposite direction, we would not have the unknot.}
\label{Fig:Annulus Twist}
\end{figure}

Theorem \ref{Annulus Twist Theorem} has a pleasing interpretation we call the ``Coiler's Conundrum''. As in Figure \ref{fig: coilerscord}, a lazy householder takes a tangled extension cord, wraps it into a coil and plugs the ends into themselves. If the resulting knot is the unknot, would it still be the unknot if one end were first wrapped through the coil $q \neq 0$ times? Theorem \ref{Annulus Twist Theorem} shows that if $|q| \geq 2$ and the original cord was sufficiently tangled, the resulting knot will not be the unknotted. Figure \ref{Fig:Annulus Twist} shows that if $q = \pm 1$, it is possible to obtain the unknot.

\begin{figure}[ht!]
\centering
\includegraphics[scale=0.7]{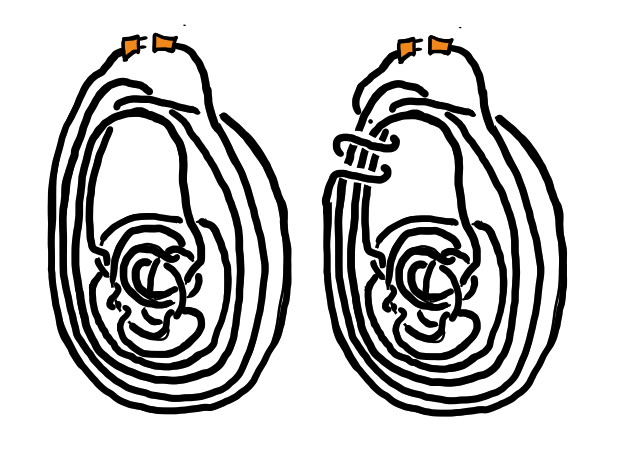}
\caption{The Coiler's Conundrum: could both be unknots?}
\label{fig: coilerscord}
\end{figure}

\subsection{Notation}

We work in the PL or smooth category. We let $I = [0,1] \subset \R$. All surfaces are compact and orientable; points of intersection with a 1-manifold will be termed \defn{punctures}. The notation $\nbhd(X)$ will denote an open regular neighborhood of some $X$ embedded in some other space $Y$, which we will always specify when there is the possibility of ambiguity. The notation $\interior(X)$ denotes the interior of a manifold $X$. For a knot $K$, $g(K)$ denotes its Seifert genus.

If $\delta$ is a properly embedded arc in an annulus or pair-of-pants $A$, we say it is \defn{spanning} if the endpoints of $\delta$ are on distinct components of $\boundary A$ and \defn{returning} if they are on the same component. A returning arc in an unpunctured annulus is necessarily inessential.

If $U$ is a 3-manifold with boundary and if $c \subset \boundary U$ is a simple closed curve, we let $U[c]$ denote the result of attaching a 2-handle to $\boundary U$ along $c$. The cocore of the 2-handle is a properly embedded arc in $U[c]$; we will always denote that arc by $\beta$. If $X \subset U$ is a properly embedded surface such that some components $\boundary_c U$ are curves parallel to $c$, in $U[c]$ we can cap them off with discs (each intersecting the cocore $\beta$ exactly once). We let $X[c]$ denote the capped off surface. When the curve $c$ is nonseparating, we will also have occasion to consider a curve $c^* \subset \boundary U$ cutting off a genus one subsurface of $\boundary U$ containing $c$. If the surface $X \subset U$ also has boundary components $\boundary_{c^*} X$ parallel to $c^*$. We will generally need to be very careful in how we handle the curves $\boundary_{c^*} X$. 

Our proof (beginning in Section \ref{final proof}) relies heavily on the author's Theorem \ref{Taylor thm} applying sutured manifold theory to questions of 2-handle addition. We will assume a small amount of familiarity with sutured manifold theory terminology as developed in \cites{Gabai1, Gabai2, Gabai3, Scharlemann}, in particular the definitions of ``taut sutured manifold'', ``norm-minimizing'', ``conditioned decomposing surface'', and ``taut hierarchy''. In this paper we only apply previous sutured manifold theory theorems; we do not need to develop any of the theory. We also need the definition of ``Thurston norm'' \cite{Thurston-norm} which we provide for convenience:

\begin{definition}
If $S$ is a compact connected orientable surface, the \defn{Thurston norm} $x(S)$ of $S$ is equal to $\max(0, -\chi(S))$. The Thurston norm of a compact orientable surface is the sum of the Thurston norms of its components. If $M$ is a compact, orientable 3-manifold and if $U \subset \boundary M$ and $y \in H_2(M,U; \Z)$ is a nontrivial class, the Thurston norm $x(y)$ of $y$ is the minimum of $x(S)$ over all compact, orientable surfaces representing $y$. We extend $x$ to a seminorm on $H_2(M,\boundary M;\R)$. The Thurston norm $x(M)$ of $M$  is the minimal Thurston norm of any non-null homologous properly embedded surface in $M$. 
\end{definition}

\subsection{Outline}

In Sections \ref{sec: 2-tangle} and \ref{sec: 1-tangle} we provide background on 2-tangles and 1-tangles, respectively. The material on 2-tangles is standard, but some of the material on 1-tangles is new. In particular, we define a ``wrapping index'' $\omega(\phi_1, \phi_2)$ between two complementary 1-tangles (requiring $\phi_1$ to be trivial). This measure is implicit in \cite{Abernathy}, but we consider some of its properties in some detail. Section \ref{example} explains the construction of inequivalent complementary 1-tangles giving unknot closures of a nontrivial 1-tangle. Section \ref{annulus twists} exhibits a connection between certain unknotted annulus twists and 1-tangle closures. Theorem \ref{Annulus Twist Theorem} is proved from Theorem \ref{Main Thm}. Section \ref{basic props} establishes some basic properties of two trivial complementary 1-tangles $(W, \phi_1)$ and $(W, \phi_2)$ where $\omega(\phi_1, \phi_2) \leq 1$ and proves a version of Theorem \ref{Main Thm}  wrapping index is at most one. Section \ref{band sums} concerns another special case: when, after pushing the arc $\phi$ of a complementary 1-tangle giving an unknot closure into $V$, the unknot $K(\phi)$ bounds a disc lying entirely in $V$. The proof is accomplished by appealing to classical results on band sums. In Section \ref{final proof}, the proof of Theorem \ref{Main Thm} is completed by examining the case when $\omega(\phi_1, \phi_2) \geq 2$.  We then conclude by showing that the Krebes tangle is persistent.

\section{Background on 2-tangles}\label{sec: 2-tangle}

We briefly review some standard terminology and background. 

\begin{definition}
A pair $(B, \tau)$ is a \defn{2-tangle}\footnote{many authors would say just ``tangle,'' but some would say``4-tangle''} if $B$ is a 3-ball and $\tau \subset B$ is a properly embedded pair of arcs. Two 2-tangles are \defn{equivalent} if they are homeomorphic as pairs. A 2-tangle is \defn{trivial} if the arcs are isotopic (relative to their endpoints) into $\boundary B$. We also say that a component $\tau_0 \subset \tau$ is \defn{unknotted}, if it is isotopic (relative to its endpoints) into $\boundary B$ (ignoring the other component). Two tangles $(B, \tau)$, and $(B, \kappa)$ are \defn{rationally equivalent} if $\tau$ is isotopic to $\kappa$ by a proper isotopy relative to $\boundary B$. An equivalence class of a trivial tangle under rational equivalence is a \defn{rational tangle}.
\end{definition}

At least since Conway's enumeration of knots of low crossing number \cite{Conway}, 2-tangles and their closures have played a central role in understanding knots.   Famously, for a given 3-ball with four marked points in its boundary, Conway used $\Q \cup \{\infty\}$ to parameterize rational tangles whose endpoints are those four marked points. Figure \ref{2-tangles} depicts a 2-tangle and a closure.

\begin{definition}
Suppose $(B, \tau)$ is a 2-tangle with $B \subset S^3$ and let $B'$ be the closure of the complement of $B$. If $(B', \tau')$ is a 2-tangle with $\boundary \tau' = \boundary \tau$, we say that $(B', \tau')$ is a \defn{complementary 2-tangle} to $(B, \tau)$ and that the knot or 2-component link $\tau \cup \tau'$ is a \defn{closure} of $(B, \tau)$. If $(B', \tau')$ is a (representative of a) rational tangle, then $\tau \cup \tau'$ is a \defn{rational closure} of $(B, \tau)$.
\end{definition}

\begin{figure}[ht!]
\labellist
\small\hair 2pt
\pinlabel {$(B,\tau)$} at 171 56
\pinlabel {$(B', \tau')$} [bl] at 308 290
\endlabellist
\centering
\includegraphics[scale=0.4]{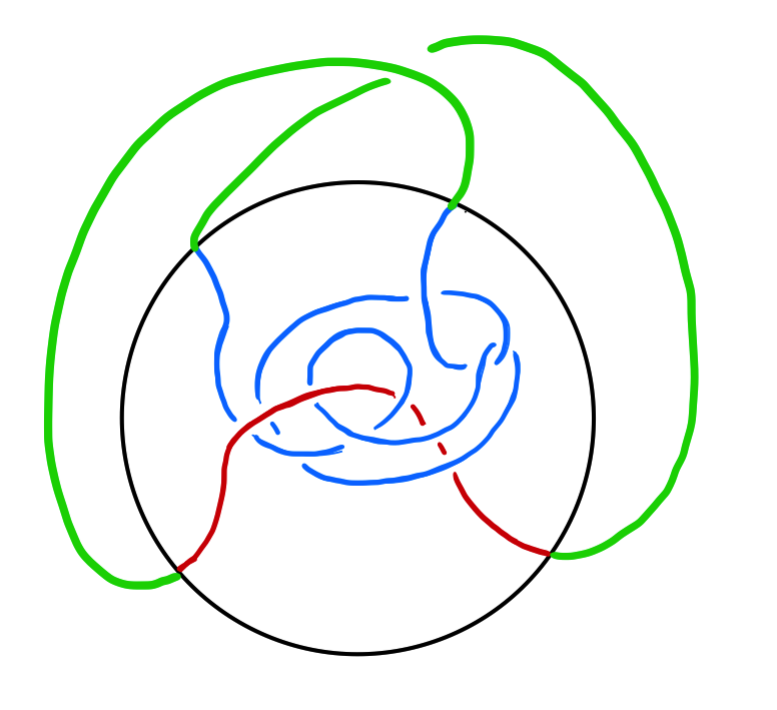}
\caption{The blue and red arcs $\tau$ are contained in a 3-ball $B$ whose boundary sphere is drawn as a black circle, making $(B, \tau)$ a 2-tangle. The green arcs $\tau'$ are contained in complementary 3-ball $B' = B\setminus \interior(B)$, making $(B', \tau')$ a complementary 2-tangle. The union of the blue, red, and green arcs is a rational closure of $(B, \tau)$. It is the unknot. Observe the similarity to Figure \ref{Fig:Samples}.D.}
\label{2-tangles}
\end{figure}

Drawing on the classification of lens spaces, Conway classified which rational closures of which rational tangles give rise to equivalent knots.  In particular, in terms of his parameterization of rational tangles it is completely known which closures of a specific rational tangle produce the unknot. 

\begin{definition}
A 1-manifold (possibly with boundary) properly embedded in a 3-manifold has a \defn{local knot} if it intersects some ball $B$  in a single arc that is not isotopic (in the ball) into $\boundary B$. A 2-tangle $(B, \tau)$ is \defn{prime} if it is nontrivial and $\tau$ does not contain a local knot.
\end{definition}

Bleiler and Scharlemann \cites{BS1, BS2} proved the following theorem; the statement we give is from \cite{EM}. It equivalent to the fact that strongly invertible knots have unique complements \cite{GL}, via the Montesinos trick \cite{Montesinos}. In a different direction, it is also a consequence of \cite{Taylor-2h}*{Theorem 6.5}. This result, and results inspired it (especially those relating to primeness of unknotting number one knots), have turned out to be very useful in knot theory.

\begin{theorem}[Bleiler-Scharlemann]\label{2-tangle thm}
Suppose that $(B, \tau)$ is a prime 2-tangle, then, up to isotopy relative to $\boundary B$, there is at most one complementary 2-tangle $(B', \tau')$ such that $\tau \cup \tau'$ is the unknot; such a complementary 2-tangle must be rational.
\end{theorem}

Krebes \cite{Krebes} raised the problem of finding examples of 2-tangles which are \defn{persistent}, that is, which do not admit \emph{any} unknot closure. The quest to identify persistent 2-tangles was taken up by many authors (for example \cites{SiWi, SiWi2, KSW, PSW, Ruberman, CR, AERSS, KL}), using a variety of techniques including double branched covers, coloring theory, skein theory, and higher order linking invariants.  The Krebes' problem concerns prime 2-tangles with \emph{no} unknot closures, whereas the Bleiler-Scharlemann theorem concerns prime 2-tangles with \emph{with at least one} unknot closure; it shows that there is then exactly one.

\section{1-tangles}\label{sec: 1-tangle}

\subsection{Terminology and preliminary results}
\begin{definition}
    A \defn{1-tangle} $(V, \psi)$ consists of an unknotted solid torus $V$ embedded in $S^3$ and a properly embedded arc $\psi \subset V$. We consider two 1-tangles to be equivalent if they are related, in $V$, by an ambient proper isotopy of the arcs, relative to their endpoints. A 1-tangle $(V, \psi)$ is \defn{trivial} if $\psi$ is isotopic in $V$, relative to $\boundary \psi$, into $\boundary V$. 
 \end{definition}
 
We fix the following notation for the remainder of the paper.
 
 \begin{definition}
 Let $(V,\psi)$ be a fixed 1-tangle. Let $S = \boundary V$ and $\circ{S} = S \setminus \nbhd(\boundary \psi$); the surface $S$ is a torus and $\circ{S}$ is a genus 1 surface with two boundary components. Let $N = V \setminus \nbhd(\psi)$. A \defn{meridian of $V$} is an essential simple closed curve in $S$ bounding a disc in $V$. A \defn{meridian of $\psi$} is an essential simple closed curve in the annulus $\boundary N  \setminus \circ{S}$. A \defn{preferred longitude} is an essential nonseparating simple closed curve $\lambda$ in $\circ{S}$ such that $\lambda$ bounds a disc in $W$. (Any two preferred longitudes are isotopic in $S$; though they may not be isotopic in $\circ{S}$.)

Let $W$ be the closure of the complement of $V$ in $S^3$; it is a solid torus. A 1-tangle $(W, \phi)$ with $\boundary \phi = \boundary \psi$ is a \defn{complementary 1-tangle} to $(V, \psi)$. The knot $K(\phi) = \phi \cup \psi$ is a \defn{closure} of $(V, \psi)$. 
 \end{definition}
 
 Figure \ref{Fig:Samples} depicts some choices for 1-tangles $(V,\psi)$ and complementary 1-tangles $(W,\phi)$, all giving unknot closures. 
 
 \begin{lemma}\label{trivial}
 Suppose that $(V, \psi)$ is a nontrivial 1-tangle and that $(W,\phi)$ is a complementary 1-tangle such that $K(\phi)$ is the unknot. Then $(W, \phi)$ is trivial.
 \end{lemma}
 \begin{proof}
 Suppose, that $K(\phi)$ is the unknot. The surface $\circ{S}$ is a genus 1 surface with two meridional boundary components on the boundary of the closure of $\nbhd(K(\phi))$. Since the exterior of $K(\phi)$ is a solid torus, $\circ{S}$ is compressible into either $V\setminus \psi$ or $W\setminus \phi$. Compressing $\circ{S}$ results in a surface with a component that is an annulus. Since $K(\phi)$ is the unknot, this annulus is $\boundary$-parallel to both sides, implying that the 1-tangle on the same side as the compressing disc is trivial. Since $(V,\psi)$ is nontrivial, $(W,\phi)$ must be trivial.
\end{proof}

 In a different direction, a division of a knot by a torus in $S^3$ into a trivial 1-tangle and trivial complementary 1-tangle is a \defn{(1,1)-position}.  A knot is a \defn{(1,1)-knot} if it has a (1,1)-position but is not the unknot or a 2-bridge knot (i.e. the closure of a rational 2-tangle by a rational 2-tangle). Various authors \cites{CK, Frias, CM, CMS} have found parameterizations of (1,1)-positions, and in \cite{CK}*{Corollary 4.9} the authors give a partial answer to the question of which parameters characterize the unknot (for one of the parameterization methods). 

Our main theorem shows that any nontrivial 1-tangle admits at most two complementary 1-tangles producing unknot closures; it is a 1-tangle analogue of Bleiler and Scharlemann's Theorem \ref{2-tangle thm}.  Our proof involves reframing the problem as a 2-handle addition problem. Here is an example of the relevance of that perspective.

\begin{lemma}\label{Jaco lem}
Suppose that $(V, \psi)$ is nontrivial and that $\psi$ does not have a local knot. Then $N$ is irreducible and $\boundary$-irreducible.
\end{lemma}
\begin{proof}
The fact that $N$ is irreducible follows immediately from the facts that $V$ is a solid torus and $\psi$ is an arc with endpoints on $S$. The surface $\circ{S}$ is incompressible in $N$; for otherwise, we could compress $S$ to a sphere in $V$ intersecting $\psi$ twice-transversally. Since $\psi$ does not contain a local knot, it is isotopic in the 3-ball $B \subset V$ bounded by that sphere into the sphere. We may arrange for it to miss the scars from the compression, and so $\psi$ would be isotopic in $V$ into $\boundary V$. But this contradicts the fact that $(V,\psi) $ is nontrivial. Thus, $\circ{S}$ is incompressible in $N$. 

Let $\mu \subset \boundary N$ be a meridian of $\psi$; it does not bound a disc in $N$. Since $\circ{S}$ is incompressible in $N$, so is $\boundary N \setminus \mu$. Since $N[\mu] = V$ is a solid torus (and is, therefore, $\boundary$-reducible), the contrapositive of the Jaco 2-handle addition lemma \cite{Jaco}*{Theorem 2} implies that $N$ is $\boundary$-irreducible.
\end{proof}

\begin{remark}
     Since $V$ is a solid torus in $S^3$, there is a unique isotopy class of essential closed curve on $S$ that bounds a compact, orientable surface in $W$. We call a curve in this isotopy class a \defn{preferred longitude} for $V$; this term will always be used reference to $V$ and never to $W$, so we will usually say simply ``preferred longitude''.  (For a closure $K(\phi)$, we do sometimes also need to refer to the preferred longitude on the boundary of the the closure of $\nbhd(K(\phi))$.) Although there is a unique isotopy class of preferred longitudes in $S$, there is not in $\circ{S}$. On the other hand, given an embedded arc $\epsilon \subset S$ joining the points $\boundary \psi$ and with interior disjoint from $\boundary \psi$, the surface $S\setminus \epsilon$ is a once-punctured torus. Thus, if $c_1, c_2$ are two essential simple closed curves in $S\setminus \epsilon$ that are isotopic in $S$, then they are also isotopic in $S\setminus \epsilon$. In what follows we will frequently apply this observation to preferred longitudes of $V$ and to simple closed curves in $S$ bounding discs containing $\epsilon$.
\end{remark}

\subsection{Wrapping index}\label{sec:wrapping index}

We use a variation of wrapping number to compare two complementary 1-tangles. See Figure \ref{fig: wrapping example} for some examples. We begin with a helpful technical lemma.

\begin{lemma}\label{disc exists}
Suppose that $(W, \phi)$ is a trivial complementary 1-tangle. Then, up to isotopy in $\circ{S}$, there is a unique preferred longitude bounding a disc in $W$ disjoint from $\phi$. Any two such discs are isotopic rel $\boundary$ in $W\setminus \phi$.
\end{lemma}
\begin{proof}
This is an innermost circle/outermost arc argument. Let $D \subset W$ be an embedded disc with boundary the union of $\phi$ and an arc $\epsilon \subset S$. Such a disc exists since $(W,\phi)$ is trivial. Let $\lambda \subset \circ{S}$ be a preferred longitude bounding a disc $\Delta \subset W$. Choose $\lambda$ to minimize $|\lambda \cap \epsilon|$. Subject to that, choose $\Delta$ to minimize $|D \cap \Delta|$, up to isotopy in $W$. The intersection $D \cap \Delta$ consists of arcs and circles. Since the interior of $D$ is disjoint from $\phi$ and $W\setminus \phi$ is irreducible, an innermost circle argument shows that the intersection consists only of arcs. Let $\zeta \subset D \cap \Delta$ be an arc of intersection that is outermost in $D$ and has endpoints on $\epsilon$; it bounds a subdisc $E \subset D$ with a subarc of $\boundary D$ and with interior disjoint from $\Delta$. Boundary-compress $\Delta$ using $D$. This converts $\Delta$ into two discs, at least one of which $\Delta$ must be essential in $W$. Its boundary is then a preferred longitude contradicting our choice of $\lambda$. Thus, no such arc $\zeta$ exists. If there is an arc $\zeta \subset D \cap \Delta$ with endpoints on $\phi$, we can use an outermost disc of $D \setminus \Delta$ to isotope $\Delta$, rel $\lambda$, to reduce $|\Delta \cap \phi|$, another contradiction. Thus, each arc of $\Delta \cap D$ has one endpoint on $\phi$ and one endpoint on $\Delta$. Using an outermost disc of $D \setminus \Delta$, we may isotope $\lambda$ and $\Delta$, sliding $\lambda$ over an endpoint of $\phi$ to eliminate an arc of intersection. This again contradicts our choice of $\lambda$.  Hence, $\Delta \cap D = \nil$, as desired.

Suppose that $\lambda'$ is another preferred longitude bounding a disc $\Delta' \subset W \setminus \phi$. In $W \setminus \phi$,  isotope $\lambda'$ and $\Delta'$ so as to minimize $|\Delta' \cap \Delta|$. Since both $\Delta$ and $\Delta'$ are unpunctured and since $W \setminus \phi$ is irreducible, there are no circles of intersection. Suppose that $\zeta \subset \Delta' \cap \Delta$ is an outermost arc of intersection in $\Delta$, cobounding a subdisc of $\Delta$ with an arc of $\lambda$ and with interior disjoint from $\Delta'$. Boundary-compress $\Delta'$ using that outermost disc. We create two discs $\Delta$ and $\Delta_2$ which, after a small isotopy, are disjoint from $\Delta'$. One of them, say $\Delta$ is essential in $W$. The curves $\boundary \Delta$ and $\boundary \Delta$ must be isotopic in $S$, as both are preferred longitudes. Thus, $\boundary \Delta_2$ is inessential in $S$. If $\boundary \Delta_2$ also bounds a disc in $\circ{S}$, we contradict the minimality of $|\Delta' \cap \Delta|$. Thus, $\boundary \Delta_2$ bounds a disc containing $\epsilon$. However, since $S$ is a torus, the curves $\boundary \Delta$ and $\boundary \Delta$ are still isotopic in $\circ{S}$. Consequently, $\Delta \cap \Delta' = \nil$ and $\boundary \Delta'$ is isotopic to $\lambda$ in $\circ{S}$, as claimed. Since $W\setminus \phi$ is irreducible, once that isotopy has been performed, the discs $\Delta'$ and $\Delta$ are also isotopic, rel $\boundary$.
\end{proof}

\begin{definition}
Suppose that $(W,\phi_1)$ and $(W, \phi_2)$ are two complementary 1-tangles, such that $(W, \phi_1)$ is trivial. Let $D \subset W$ be an essential disc disjoint from $\phi_1$. The \defn{wrapping index} between $\phi_1$ and $\phi_2$ is 
    \[
    \omega(\phi_1, \phi_2) = \min\limits_{\phi'_2, D} |D \cap \phi'_2|
    \]
where the minimum is taken over all complementary 1-tangles $(W, \phi'_2)$ equivalent to $(W, \phi_2)$ and essential discs $D \subset W$ disjoint from $\phi_1$ and transverse to $D$.
\end{definition}

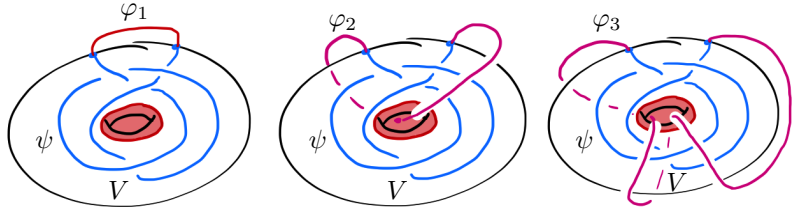
\begin{figure}[ht!]
\labellist
\small\hair 2pt
\pinlabel {$V$} [b] at 94 19
\pinlabel {$V$} [b] at 304 19
\pinlabel {$V$} [b] at 515 25
\pinlabel {$\psi$} [r] at 49 67
\pinlabel {$\psi$} [r] at 255 67
\pinlabel {$\psi$} [r] at 456 67
\pinlabel {$\phi_1$} [b] at 106 154
\pinlabel {$\phi_2$} [b] at 263 146
\pinlabel {$\phi_3$} [b] at 462 146
\endlabellist
\centering
\includegraphics[scale=0.5]{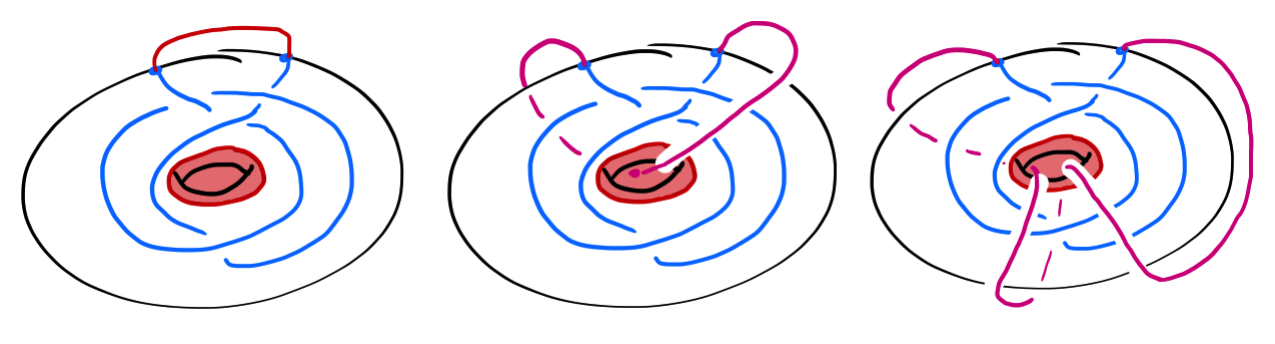}
\caption{From left to right, there are the closures of a 1-tangle $(V, \psi)$ by a complementary 1-tangles $(W, \phi_i)$, for $i = 1,2,3$ in order. The indicated disc $D$ whose boundary is a preferred longitude and which is disjoint from $\phi_1$, realizes the fact that $\omega(\phi_1, \phi_2) = 1$ and $\omega(\phi_1, \phi_3) = 2$.}
\label{fig: wrapping example}
\end{figure}

Lemma \ref{disc exists} shows that wrapping index is well-defined and that if $(W, \phi)$ is trivial, then $\omega(\phi, \phi) = 0$

\begin{remark}
Observe that in our definition of wrapping number, the arc $\phi_2$ is allowed to be isotoped through the arc $\phi_1$; that is, we use isotopies of $\phi_2$, not $\phi_1 \cup \phi_2$. Although it is possible to define a stronger version of wrapping number that uses isotopies of $\phi_1 \cup \phi_2$, we will not need it.
\end{remark}

Here is a partial converse to Lemma \ref{disc exists}.

\begin{lemma}\label{lem: wrapping well def}
Suppose that, for $i = 1,2$, $(W, \phi_i)$ is a complementary 1-tangle, with $(W, \phi_1)$ trivial. If $\omega(\phi_1, \phi_2) = 0$, then the following hold:
\begin{enumerate}
\item If $(W,\phi_2)$ is trivial, then $(W, \phi_1)$ and $(W, \phi_2)$ are equivalent.
\item If $(W, \phi_2)$ is not trivial, it is obtained from $(W, \phi_1)$ by the addition of local knots.
\end{enumerate}
\end{lemma}
\begin{proof}
Let $D \subset W$ be an essential disc disjoint from $\phi_1$ and $\phi_2$. Compressing $S$ using $D$ results in a sphere $P$ containing the endpoints of the arcs.  The sphere $P$ contains two discs that are the scars from $D$. If $\phi_2$ is trivial, then it is also unknotted in the ball $B$ bounded by $P$ and so is isotopic rel $\boundary \phi_2$ in $B$ (and hence in $W$) to $\phi_1$. If $\phi_2$ is knotted in $B$, then we see that it contains a local knot. Surgering along $P$ converts $\phi_2$ into an unknotted arc which must then be isotopic rel $\boundary \phi_2$ to $\phi_1$.
\end{proof}

We use the next example in the proof of Theorem \ref{Annulus Twist Theorem}. 

\begin{lemma}\label{coilers calculation}
    Suppose that $\kappa \subset S$ is an oriented essential simple closed curve containing $\boundary \psi$. Let $\epsilon_1, \epsilon_2$ be the arcs into which $\boundary \psi$ divides $\kappa$. Let $\phi_1, \phi_2$ be the result of pushing the interiors of $\epsilon_1, \epsilon_2$ (respectively) into the interior of $W$. If $\kappa$ represents $q[\mu]$ in $H_1(W)$, where $[\mu]$ is the class of a meridian $\mu$ of $V$ and $q \in \Z$, then
    \[
    \omega(\phi_1, \phi_2) = \omega(\phi_2, \phi_1) = |q|.
    \]
\end{lemma}
\begin{proof}
    Let $\{i,j\} = \{1,2\}$. Any meridian disc for $W$ must meet $\kappa$ at least $|q|$ times; so any meridian disc for $W$ disjoint from $\phi_i$ must meet $\phi_j$ at least $|q|$ times. On the other hand, we may isotope $W$ so that $\epsilon_i$ becomes a very short arc. It is then easy to construct a meridian disc for $W$ disjoint from $\epsilon_i$ and intersecting $\epsilon_j$ exactly $|q|$ times. Reversing the isotopy and pushing the interiors of $\epsilon_i$ and $\epsilon_j$ into the interior of $W$ we find a meridian disc of $W$ disjoint from $\phi_i$ and intersecting $\phi_j$ exactly $|q|$ times. Hence, $\omega(\phi_i, \phi_j) = |q|$.
\end{proof}

\begin{remark}
Generalizing the previous proof, we see that if $(W,\phi_1),(W, \phi_2)$ are trivial complementary 1-tangles such that $\phi_1 \cap \phi_2 = \boundary \psi$, then $\omega(\phi_1, \phi_2)$ and $\omega(\phi_2, \phi_1)$ are both bounded below by the winding number of the knot $\kappa = \phi_1 \cup \phi_2$ in $W$. 
\end{remark}

\subsection{Converting 2-tangles to 1-tangles and 1-tangles to 2-tangles}

From a 2-tangle with one arc unknotted, we can create a 1-tangle, see Figure \ref{fig: inducing}, reading from left to right.
\begin{definition}\label{induced 1}
Suppose that $(B, \tau)$ is a 2-tangle with $B \subset S^3$ and $\tau_0 \subset \tau$ an unknotted component. Then $(V, \psi) = (B \setminus \nbhd(\tau_0), \tau \setminus \tau_0)$ is a 1-tangle, we say it is \defn{$\tau_0$-induced} from $(B, \tau)$.\end{definition}

\begin{figure}[ht!]
\labellist
\small\hair 2pt
\pinlabel {$\tau_0$} [l] at 48 64
\pinlabel {$\lambda$} [l] at 239 69
\endlabellist
\centering
\includegraphics[scale=0.5]{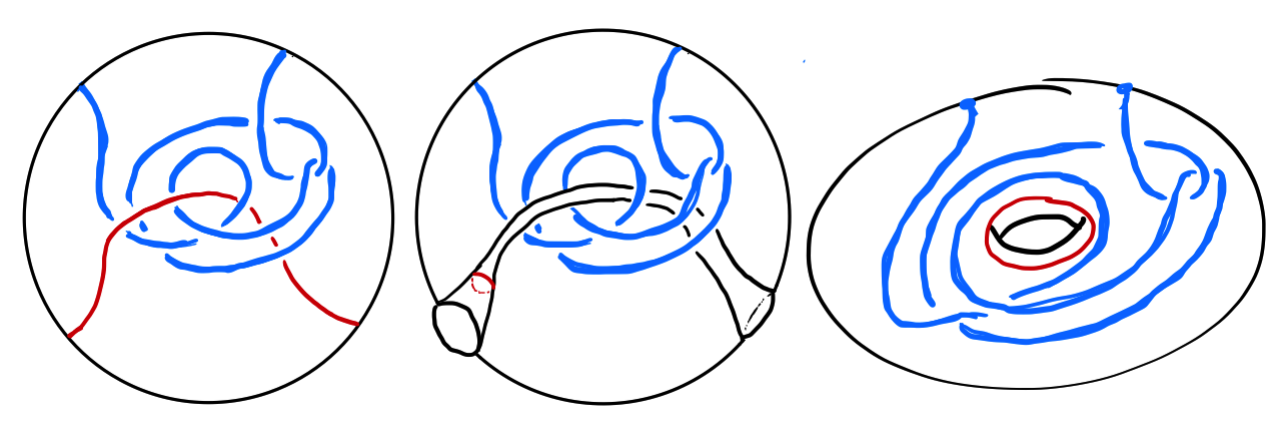}
\caption{Inducing a 1-tangle from a 2-tangle and vice versa. On the left is a 2-tangle with an unknotted arc $\tau_0$ in red. Drilling it out from the 3-ball creates the 1-tangle in the center, a preferred longitude $\lambda$ is marked, corresponding to a meridian of $\tau_0$. We can reshape the 1-tangle, preserving the endpoints of the arc and the preferred longitude, into the 1-tangle on the right. Conversely, given the 1-tangle and preferred longitude in the center (or on the right), we can attach a 2-handle along the preferred longitude to create the 2-tangle on the left, where one of the arcs is the cocore of the 2-handle.}
\label{fig: inducing}
\end{figure}

In Definition \ref{induced 1}, note that a meridian of $\tau_0$ becomes a preferred longitude for $(V, \psi)$. There is an inverse construction. See Figure \ref{fig: inducing}, reading from right to left.

\begin{definition}
Suppose that $(V, \psi)$ is a 1-tangle and that $\lambda$ is a preferred longitude. Let $B = V[\lambda]$; it is a 3-ball embedded in $S^3$. Let $\tau \subset B$ be the union of $\psi$ with the arc $\beta$ that is the cocore of the 2-handle attached to $\lambda$. Then $(B, \tau)$ is a 2-tangle. We say that $(B, \tau)$ is \defn{$\lambda$-induced} from $(V, \psi)$. 
\end{definition}

\begin{remark}
These constructions are inverses of each other. That is, if $(B, \tau)$ is a 2-tangle with $B \subset S^3$ and $\tau_0$ an unknotted component of $\tau$ and $\lambda$ a meridian of $\tau_0$, then $(B, \tau)$ is rationally equivalent to the 2-tangle that is $\lambda$-induced from the $\tau_0$-induced 1-tangle. Conversely, if $(V, \psi)$ is a 1-tangle with preferred meridian $\lambda$, then $(V,\psi)$ is equivalent to the $\beta$-induced 1-tangle that is $\lambda$-induced from $(V,\psi)$, where $\beta$ is the cocore of the 2-handle. 

Actually, there is a small ambiguity in those observations:  the choice of preferred longitude is made up to isotopy in $\circ{S}$ and there are different choices of regular neighborhoods. These differences will be immaterial in what follows.
\end{remark}

\begin{lemma}\label{lem: nontriv}
    Suppose $(B, \tau) = (V[\lambda], \psi \cup \beta)$ is a $\lambda$-induced 2-tangle from $(V,\psi)$. If $(V,\psi)$ is nontrivial , so is $(B,\tau)$.
\end{lemma}
\begin{proof}
  We prove the contrapositive. Assume $(B, \tau)$ is trivial. Then there is a properly embedded disc $D$ in $B$ separating the strands of $\tau$. The disc $D$ is then a compressing disc for $\circ{S}$ in $V \setminus \psi$. The curve $\boundary D$ is separating in $\circ{S}$, so it bounds a twice punctured disc $E \subset S$.  The sphere $D \cup E$ bounds a 3-ball in $B$ containing $\psi$. Since $\psi$ does not have a local knot (as $(B,\tau)$ is trivial) $\psi$ is isotopic into $E$ relative to its endpoints. Hence, $(V, \psi)$ is trivial.
\end{proof}

\section{Examples}\label{example}

In this section we give a construction, using Brunnian $\theta$-curves, of nontrivial 1-tangles having two inequivalent unknot closures of wrapping index 1 from each other. We then show that this construction accounts for all such examples.

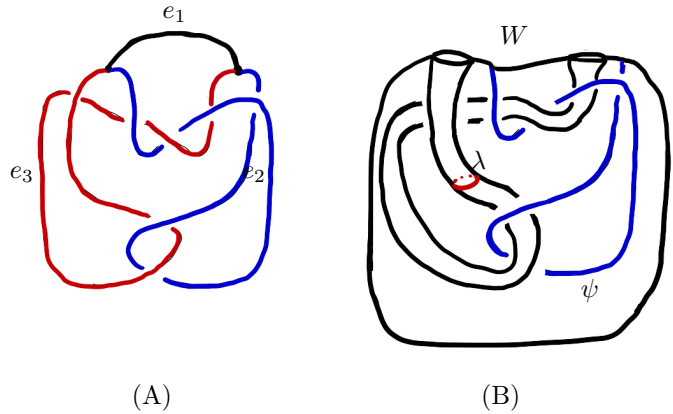
\begin{figure}[ht!]
\labellist
\small\hair 2pt
\pinlabel {$e_1$} [b] at 124 280
\pinlabel {$e_3$} [r] at 22 169
\pinlabel {$e_2$} [r] at 197 169
\pinlabel {$\psi$} [t] at 437 94
\pinlabel {$W$} at 380 274
\pinlabel {$\lambda$} [b] at 353 169
\pinlabel {(A)} [t] at 107 15
\pinlabel {(B)} [t] at 370 15
\endlabellist
\centering
\includegraphics[scale=0.5]{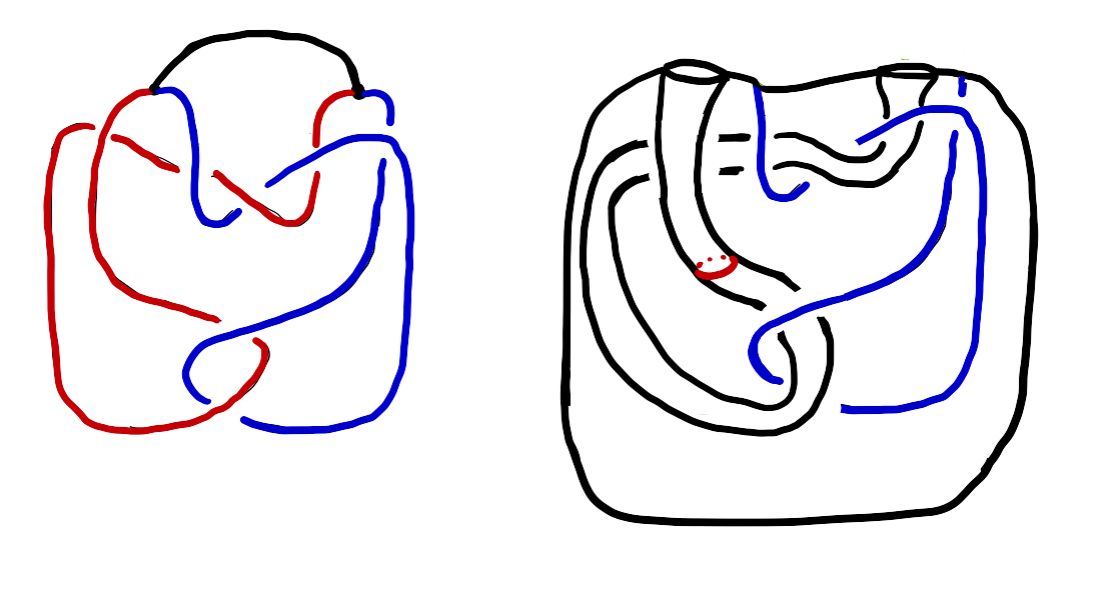}
\caption{Figure (A) shows the Kinoshita graph. Drilling out the unknot $K_{13}$ creates a 1-tangle having two different unknot closures.}
\label{fig: Kinoshita}
\end{figure}

A \defn{$\theta$-curve} is a spatial graph in $S^3$ with exactly two vertices and three edges $e_1, e_2, e_3$, each of which joins the two vertices. Let $K_{ij} = e_i \cup e_j$ for $1 \leq i < j \leq 3$; these are the \defn{constituent knots} of the $\theta$-curve. A $\theta$-curve is \defn{Brunnian} if it is not isotopic into a Heegaard sphere in $S^3$ and if each of its constituent knots is unknotted. Consider a Brunnian $\theta$-curve $\Gamma$. Figure \ref{fig: Kinoshita}.A shows the example of the Kinoshita graph \cite{Kinoshita}, the best known Brunnian $\theta$-curve.

Let $V$ be a regular neighborhood of $K_{13}$ and let $W$ be the complementary solid torus. Observe these give a genus 1 Heegaard splitting of $S^3$. Let $\psi = e_2 \cap V$, so that $(V, \psi)$ is a 1-tangle. Let $\phi_1 = e_1 \cup (e_2 \cap W)$ and $\phi_3 = (e_2 \cup e_3) \cap W$, so that $(W, \phi_2)$ and $(W, \phi_3)$ are complementary 1-tangles. The closures $K(\phi_1) = K_{12}$ and $K(\phi_3) = K_{23}$ are then both unknots. Letting $\lambda \subset S = \boundary V$ be a meridian of $e_3$, we see that $\omega(\phi_1, \phi_3) = 1$. 

 Let $B'$ be a closed regular neighborhood of $e_1$. Let $B$ be the closure of its complement. Then $(B, B \cap K_{23})$ and $(B', B' \cap K_{23})$ are each 2- tangles. See Figure \ref{fig: Kinoshita}.B. Note that each strand of $B \cap K_{23}$ is unknotted in $B$ (ignoring the other strand), as $K_{23}$ is an unknot. Let $V = B \setminus \nbhd(e_3)$. Since $B \cap e_3$ is unknotted in $B$, $V$ is an unknotted solid torus in $S^3$. Let $\psi = e_2 \cap V$. A meridional curve $\lambda$ of $e_3$ is a preferred longitude for $V$. Indeed, $(B, B \cap K_{23})$ is the $\lambda$-induced 2-tangle from $(V, \psi)$. See Figure \ref{fig: Kinoshita}.B. It is easily checked that in a Brunnian $\theta$-curve, no cycle bounds an edge disjoint from its complementary edge. Ozawa and Tsutsumi \cite{OT}, showed that the exterior of a Brunnian $\theta$-curve is $\boundary$-irreducible, so the 1-tangle $(V,\psi)$ in our construction is nontrivial. We have constructed the 1-tangles we were looking for.

\begin{figure}[ht!]
\labellist
\small\hair 2pt
\pinlabel {$\phi_1$} [b] at 172 309
\pinlabel {$\phi_3$} [b] at 414 290
\pinlabel {$(V,\psi)$} at 142 65
\pinlabel {$(V,\psi)$} at 439 65
\endlabellist
\centering
\includegraphics[scale=0.5]{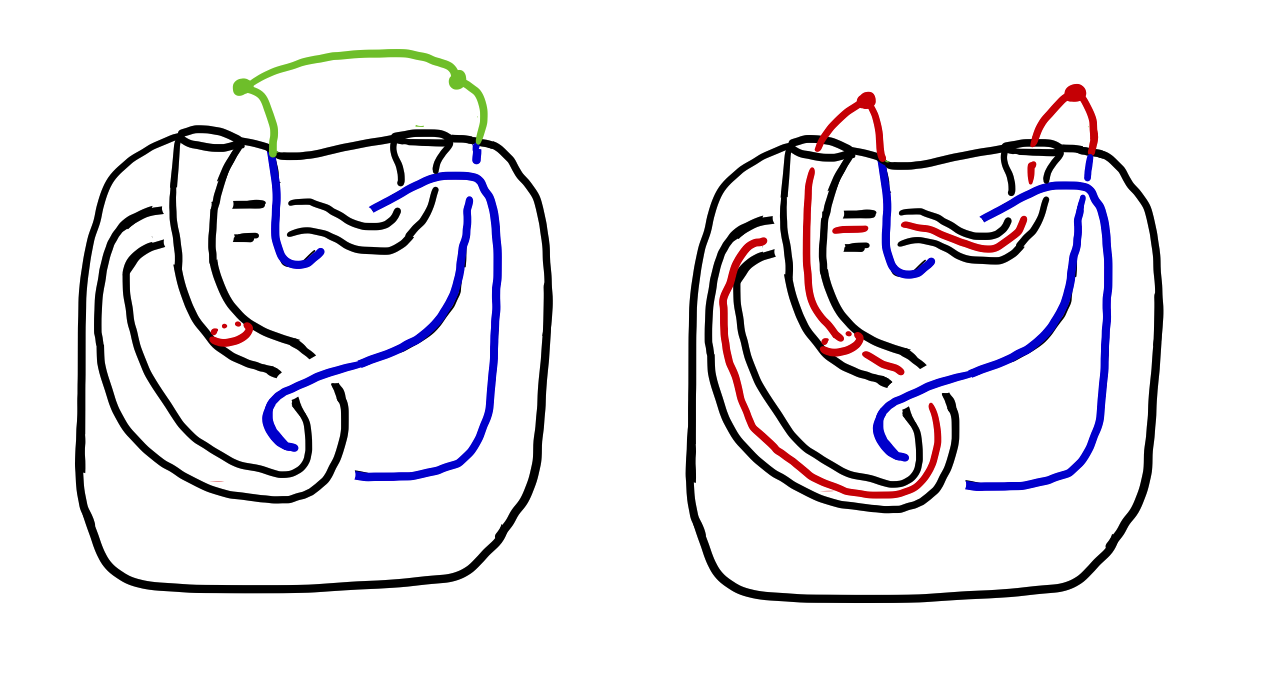}
\caption{Two different unknot closures $(W, \phi_1)$ and $(W, \phi_2)$ of the 1-tangle $(V,\psi)$ in Figure \ref{fig: Kinoshita}.B.  Both complementary 1-tangles are constructed from the other edges of the Kinoshita graph.}
\label{fig: Kinoshitaclosure}
\end{figure}
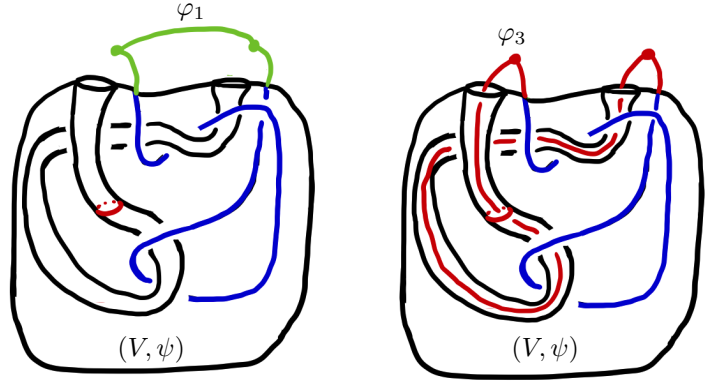

A number of papers, eg. \cites{SW, JKLMTZ, BPV}, have constructed Brunnian $\theta$-curves of increasing complexity, so there are many examples of nontrivial 1-tangles having two inequivalent unknot closures that are of wrapping index 1 apart. 

Conversely, suppose that $(V,\psi)$ is a nontrivial 1-tangle and that $(W, \phi_1)$ and $(W, \phi_3)$ are two complementary 1-tangles giving unknot closures of $(V, \psi)$. Recall from Lemma \ref{trivial} that both $\phi_1$ and $\phi_3$ are trivial. Assume that  $\omega(\phi_1, \phi_2) = 1$. Let $\lambda \subset S$ be a preferred longitude bounding a disc disjoint from $\phi_1$ and intersecting $\phi_3$ in one point. Let $\beta \subset N[\lambda]$ be the cocore of a 2-handle attached along $\lambda$. Since $S[\lambda]$ is a sphere that is the boundary of a regular neighborhood $B'$ of $\phi_1$, we may extend the endpoints of $\beta$ through that neighborhood (using a product structure on $B'\setminus \phi_1$) so as to lie on $\phi_1$. Then $\Gamma = K(\phi_1) \cup \beta$ is a $\theta$-curve. Its exterior is $N$, which by Lemma \ref{Jaco lem} is $\boundary$-irreducible, so $\Gamma$ is nontrivial. The constituent knots of $\Gamma$ are the core of $W$, a knot isotopic to $K(\phi_1)$, and a knot isotopic to $K(\phi_3)$, so $\Gamma$ is Brunnian.

\section{Annulus twists}\label{annulus twists}

\begin{definition}\label{def: annulus twist}
Suppose that $K \subset S^3$ is a knot and that $A \subset S^3$ is an embedded oriented annulus with $\boundary A \cap K = \nil$, and the interior of $A$ transverse to $K$. Give $\boundary A$ the orientation induced from the orientation of $A$. If $A \cap K \neq \nil$ and if no component of $\boundary A$ bounds a disc in $S^3 \setminus K$, we call $A$ a \defn{twisting annulus} for $K$. If $A$ lies in a 2-sphere  $P \subset S^3$, it is \defn{unknotted}. The unknotted twisting annulus $A$ is \defn{sufficiently incompressible} if there is such a $P$, such that every compressing disc $D$ for $P \setminus K$ in $S^3 \setminus K$ has $|\boundary D \cap \boundary A| > 2$.

If $A$ is a twisting annulus for $K$ and $q \in \Z$, we let the knot $K_q$ be obtained by performing $1/q$ Dehn surgery on each component of $\boundary A$; that is by $q$-twists of $K$ around $A$; with the direction of twisting determined by the sign of $q$ with respect to the orientation of $A$.
\end{definition}

\begin{remark}
Some authors prefer to orient the components of $A$ in the same direction and then perform $1/q$ surgery on one component and $-1/q$ surgery on the other.

Note that if for each component $c \subset \boundary A$, $P \setminus (c \cup K)$ is incompressible, then $A$ is sufficiently incompressible. In particular, if $P\setminus K$ is an essential meridional planar surface , then $A$ is automatically sufficiently incompressible. 
\end{remark}

\begin{remark}
Let $c_1, c_2$ be the components of $\boundary A$. If we were to relax the definition of ``twisting annulus'' to allow one of $c_1, c_2$ to bound a disc in $S^3\setminus K$, then the knots before and after the twist are equivalent, as the annulus twist amounts to twisting the knot $K$ around a once-punctured disc. Hence, we require that neither $c_1$ nor $c_2$ bound a disc in $S^3 \setminus K$. Similarly, if we allowed $A \cap K = \nil$, then as the components of $\boundary A$ are oppositely oriented, the surgeries on the components of $\boundary A$ would cancel, and the twisting would recreate $K$. 

On the other hand, if one of the curves $c_1 \subset \boundary A$ bounds a disc in $S^3$ intersecting $K$ exactly once, then the annulus twist of $K$ around $A$ is equivalent to twisting $K$ around a disc with boundary $c_2$. By \cite{KMS}*{Theorem 4.2}, it is generally not possible for both $K$ and $K_q$ to be unknots. Assuming $c_2$ does not bound a disc disjoint from $K$, there are only two possible exceptions. In each, $K$ can be arranged to lie as a 2-braid in $S^3 \setminus \nbhd(c_2)$ and $q$ is one of $\pm 1$. There are diagrams in \cite{KMS} realizing these exceptions where  $c_2$ is the braid axis, $K$ is the closure of a braid with a single crossing, and, after the twist, $K_q$ has a nonalternating diagram with three crossings. This gives slightly more information than our Theorem \ref{Annulus Twist Theorem}, in the setting when $c_1$ bounds a once-punctured disc. 
\end{remark}

\begin{proof}[Proof of \ref{Annulus Twist Theorem} from \ref{Main Thm}]
Suppose that $K = K_0 \subset S^3$ is an unknot and that $A$ is a sufficiently incompressible unknotted twisting annulus for $K$ with $|A \cap K| = 1$.  Let $P \subset S^3$ be a 2-sphere containing $A$ realizing the definition of ``sufficiently incompressible unknotted twisting annulus.'' Let $W$ be a solid torus obtained by a slight thickening of $A$ in a normal direction, as in Figure \ref{convert}.  Then $\phi_0 = W \cap K$ is an unknotted arc. An equivalent knot to $K_q$ can be obtained by replacing $\phi_0$ with an unknotted arc $\phi_q$ having the same endpoints as $W\cap K$ but running $q$ times longitudinally around $W$ (with the sign of $q$ indicating the direction of winding).  Let $V$ be the complementary solid torus to $W$. Let $\psi = K \cap V$. To apply Theorem \ref{Main Thm}, we need only show that $(V, \psi)$ is nontrivial. 

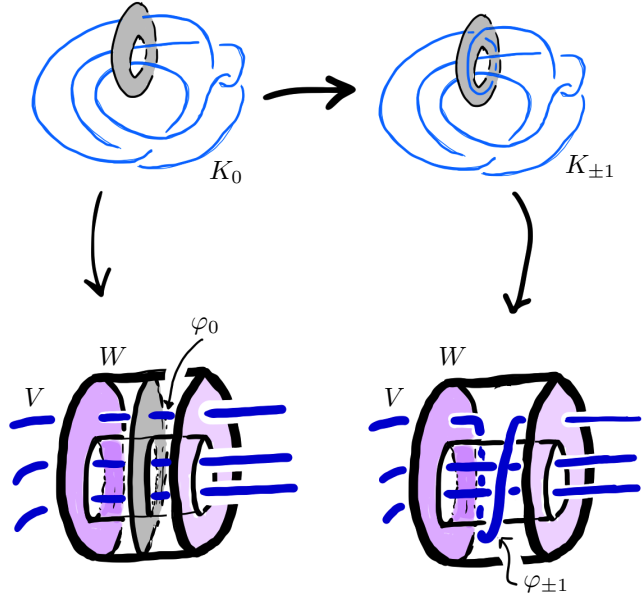
\begin{figure}[ht!]
\labellist
\small\hair 2pt
\pinlabel {$K_{0}$} [tl] at 184 391
\pinlabel {$K_{\pm 1}$} [tl] at 489 396
\pinlabel {$W$} [b] at 105 208
\pinlabel {$\phi_0$} [b] at 184 235
\pinlabel {$V$} at 38 185
\pinlabel {$W$} [b] at 396 208
\pinlabel {$V$} at 346 185
\pinlabel {$\phi_{\pm 1}$} [l] at 452 25
\endlabellist
\centering
\includegraphics[scale=0.45]{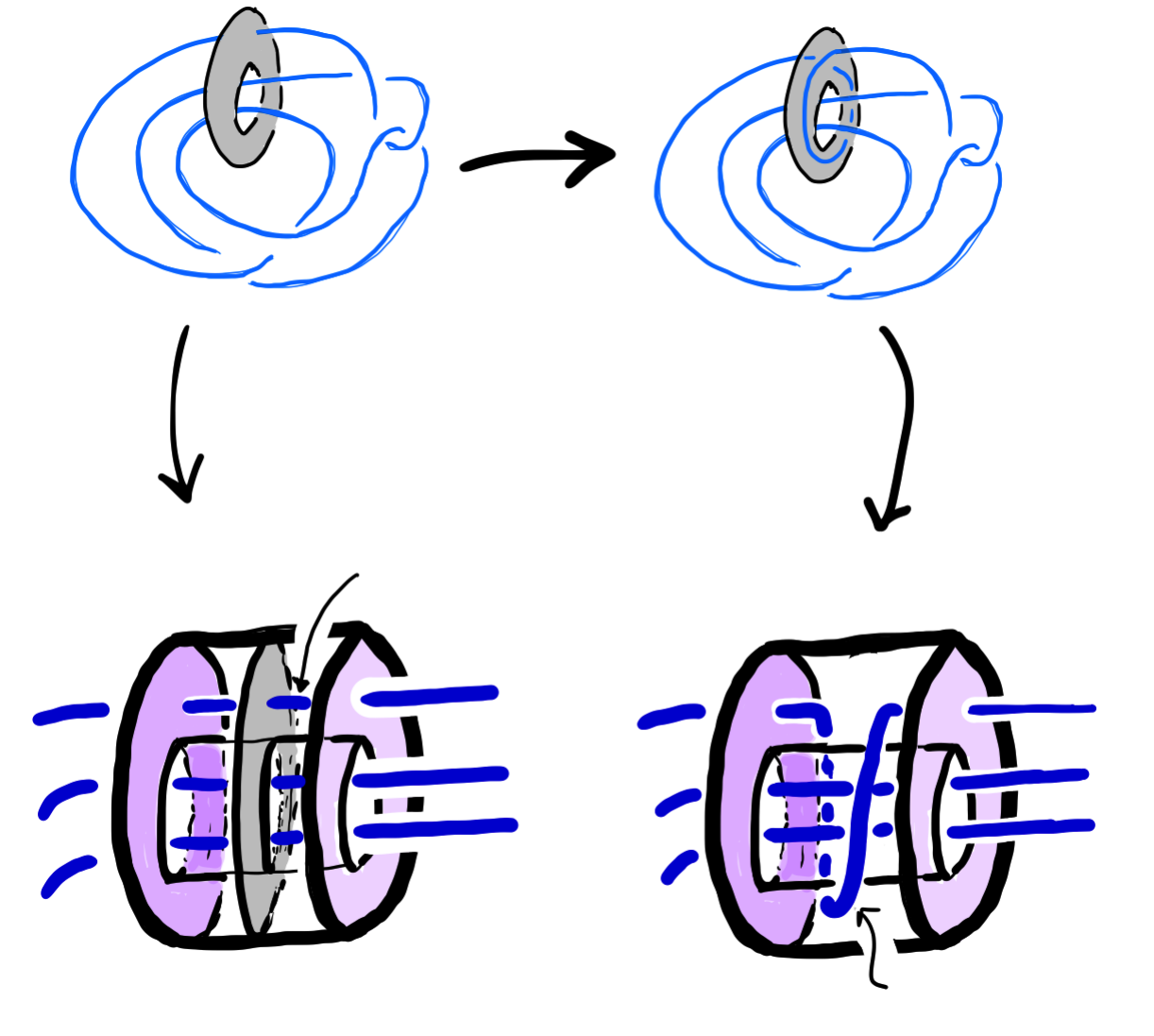}
\caption{Creating 1-tangles from an annulus twist. The solid torus $W=A \times I$, where $A$ is the twisting annulus. The knot $K$ intersects $W$ in an arc $\phi_0 = \{\text{point}\} \times I$. The result of the annulus twist can be realized by replacing $\phi_0$ with another trivial arc $\phi_q$. We depict the situation when $q = \pm 1$, with the sign depending on the orientation of the annulus.}
\label{convert}
\end{figure}

Suppose, for a contradiction, that $(V, \psi)$ is trivial. In this case, $\circ{S}$ is compressible in $N = V \setminus \nbhd(\psi)$ via some compressing disc $\Delta$. Let $P_V = P \cap V$; it is the union of two discs, each intersecting $K$. Properly isotope $\Delta$ in $V\setminus \psi$ so that $|\Delta \cap P_V|$ is minimal. The intersection $\Delta \cap P_V$ consists of circles and arcs. Since $P_V\setminus K$ is incompressible in $S^3 \setminus K$ and $K$ intersects each component of $P_V$, the intersection $\Delta \cap P_V$ consists solely of arcs. 

Let $\zeta \subset \Delta \cap P_V$ be an arc of intersection that is outermost in $\Delta$, with $D' \subset \Delta$ the outermost disc it bounds. The curve $\boundary D'$ is the union of $\zeta$ with an arc $\zeta' \subset \boundary \Delta$. The arc $\delta$ is a returning arc in the once-punctured annulus $S \setminus \boundary P_V$. Using the construction of $W$ as a normal bundle of $A$, we may use the product structure on $(W, K \cap W)$ to extend $\zeta'$ through $W\setminus K$ so as to lie on $A$. This converts the disc $D'$ into a disc $D$ which is a compressing disc for $P \setminus K$ in $S^3 \setminus K$ where $|\boundary D \cap \boundary A| = 2$. This contradicts the assumption that $A$ was sufficiently incompressible. Thus, $(V, \psi)$ is nontrivial. Theorem \ref{Main Thm} implies that, up to isotopy relative to endpoints in $W$, there are at most two properly embedded arcs $\phi \subset W$ such that $K(\phi) = \psi \cup \phi$ is the unknot. The arc $\phi_0 = K \cap W$ is one of those arcs.  Theorem \ref{Main Thm} also says that if $K_q = K(\phi_q)$ and $K_\ell = K(\phi_\ell)$ is the unknot, then $\omega(\phi_\ell, \phi_q) = \pm 1$.  The knot $\kappa = \phi_\ell \cup \phi_q$ can be isotoped (relative to $\boundary \psi$) to lie in $S$ and it represents $(\ell + q)[\mu]$ in $H_1(W)$ where $\mu$ is meridian of $V$ (i.e. preferred longitude of $V$). By Lemma \ref{coilers calculation}, $\omega(\phi_\ell, \phi_q) = |q + \ell|$. Hence, $q = \ell \pm 1$, so both $q, \ell \in \{0, -1, +1\}$ and if one of them is $\pm 1$, the other cannot be $\mp 1$. Thus, there is at most one $q \neq 0$ such that $K_q$ is the unknot, and if there is such a $q$, then $q = \pm 1$. 
 \end{proof}

\section{Wrapping index one}\label{basic props}

It turns out that when $\omega = 1$, there is a close relationship between 1-tangle closures and 2-tangle closures. See Figure \ref{distance 1} for a depiction of the next lemma. 

\begin{proposition}\label{reduction}
Suppose that $(V, \psi)$ is a 1-tangle and that $(W, \phi_i)$ is a complementary 1-tangle for $i = 1,2,3$. Suppose that $(W, \phi_1)$ is trivial and that $\omega(\phi_1, \phi_2) = \omega(\phi_1, \phi_3) = 1$. Then there is a preferred longitude $\lambda$ bounding a disc $D \subset W$ such that the following hold:
\begin{enumerate}
\item the arcs $\phi_2$ and $\phi_3$ can be isotoped in $W$, relative to their endpoints so that $D \cap \phi_2 = D \cap \phi_3$ is a single point.
\item Let $(B, \tau)$ be the $\lambda$-induced 2-tangle from $(V, \psi)$ and let $B'$ be the closure of the complement of $B$ in $S^3$. Then the 2-tangles $(B', \tau'_2) = (B', K(\phi_2) \cap B')$ and $(B', \tau'_3) = (B', K(\phi_3) \cap B')$ are rationally equivalent if and only if the 1-tangles $(W, \phi_2)$ and $(W, \phi_3)$ are equivalent. 
\end{enumerate}
\end{proposition}

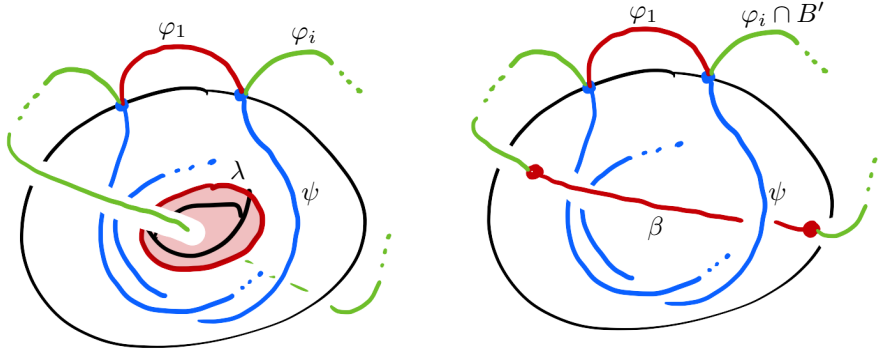
\begin{figure}[ht!]
\labellist
\small\hair 2pt
\pinlabel {$\phi_1$} [b] at 132 250
\pinlabel {$\phi_i$} [b] at 231 246
\pinlabel {$\psi$} [l] at 225 137
\pinlabel {$\phi_i \cap B'$} [b] at 592 259
\pinlabel {$\lambda$} [b] at 181 144
\pinlabel {$\phi_1$} [b] at 487 262
\pinlabel {$\beta$} [t] at 496 124
\pinlabel {$\psi$} [l] at 577 137
\endlabellist
\centering
\includegraphics[scale=0.5]{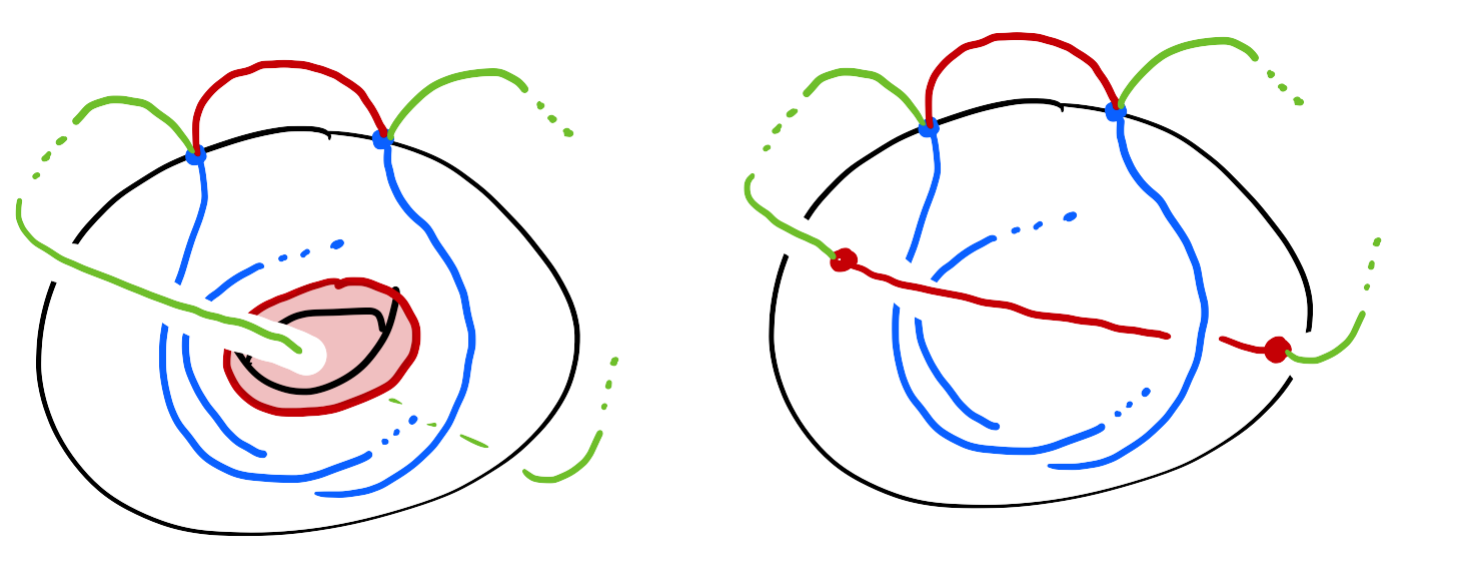}

\caption{When, for $i = 2,3$, we have $\omega(\phi_1, \phi_i) = 1$, the 1-tangle closure can be converted into a 2-tangle closure, as in Proposition \ref{reduction}. The arc $\beta$ is the cocore of the 2-handle attached along the preferred longitude $\lambda$. }
\label{distance 1}
\end{figure}

\begin{proof}
See Figure \ref{distance 1} for a depiction. By definition of $\omega$, there is an essential disc $D_1$ in $W$ disjoint from $\phi_1$ and intersecting $\phi_2$ in a single point. Similarly, there is an essential disc $D'_1$ in $W$ disjoint from $\phi_1$ and intersecting $\phi_3$ in a single point. By Lemma \ref{disc exists}, we may assume that $D_1 = D'_1$.  By an isotopy supported in a neighborhood of $D_1 = D'_1$ we may take the point $D_1 \cap \phi_3$ to the point $D_1 \cap \phi_2$. Furthermore, we may arrange that in a small regular neighborhood of $D_1$, the arc of intersection $\beta$ between $\phi_3$ and that neighborhood exactly coincides with the arc of intersection between $\phi_2$ and that neighborhood.

Let $(B, \tau)$ be the $\lambda$-induced 1-tangle from $(V, \psi)$. In the construction of $(B, \tau)$, use the aforementioned regular neighborhood of $D$ for the 2-handle and the arc $\beta$ for the cocore of the 2-handle. Let $B'$ be the closure of the complement of $B$. Note that $B' \subset W$. Let $(B', \tau'_i) = (B', B' \cap K(\phi_i))$ for $i = 2,3$. 

Suppose that $(B', \tau'_2)$ and $(B', \tau'_3)$ are rationally equivalent. This implies that there is an isotopy in $B'$, of $\tau'_3$, relative to its endpoints, taking $\tau'_3$ to $\tau'_2$. As $B' \subset W$, this induces an isotopy of $\phi_3$ to $\phi_2$, relative to its endpoints. The isotopy does not move the arc $\beta$. Hence, $(W, \phi_3)$ is equivalent to $(W, \phi_2)$.

As we do not rely on it, we simply sketch the converse. Assume $(W, \phi_3)$ is equivalent to $(W, \phi_2)$, and consider an isotopy to be realized by a (smooth or PL) map $F \co [0,1] \times [0,1] \to W$ with $F([0,1] \times \{0\}) = \phi_3$ and $F([0,1] \times \{1\}) = \phi_2$. Since $\phi_2$ and $\phi_3$ coincide along $\beta$, we may modify $F$ so that $F^{-1}(\beta)$ is $X = I \times [0,1]$ where $I \subset [0,1]$ is a closed subinterval. Standard surgery techniques then let us modify $F$ further so that the inverse image of the regular neighborhood $D$ is equal to $X$. The isotopy $F$ then restricts to a rational equivalence between $(B', \tau'_2)$ and $(B', \tau'_3)$. 
\end{proof}

\begin{lemma}\label{wrapping 1 lem}
Suppose that $(W, \phi_1)$ and $(W, \phi_2)$ are inequivalent trivial complementary 1-tangles. Then the following are equivalent:
\begin{enumerate}
\item $\omega(\phi_1, \phi_2) = 1$
\item $\omega(\phi_2, \phi_1) = 1$
\item If $a, b$ are preferred longitudes such that $a$ bounds an essential disc in $W$ disjoint from $\phi_1$ and $b$ bounds an essential disc in $W$ disjoint from $\phi_2$, then $a$ and $b$ can be isotoped in $\circ{S}$ so as to be disjoint.
\end{enumerate}
\end{lemma}
\begin{proof}
Since statement (3) is symmetric in $\phi_1$ and $\phi_2$, it suffices to show that (1) and (3) are equivalent; for the same argument would show that (2) and (3) are equivalent. Let $a, b$ be preferred longitudes bounding discs in $W$ disjoint from $\phi_1, \phi_2$ respectively. Assume that $a, b$ have been isotoped in $\circ{S}$ so as to intersect minimally. 

Assume that $\omega(\phi_1, \phi_2) = 1$; we will show $a \cap b = \nil$. Perhaps changing $\phi_2$ by an isotopy in $W$ relative to its endpoints,  there exists an essential disc $D$ in $W$ such that $D \cap \phi_1 = \nil$, $|D \cap \phi_2| = 1$, and $a = \boundary D$. Let $E \subset W$ be an essential disc in $W$ disjoint from $\phi_2$ with $\boundary E = b$. Out of all such, assume that $|E \cap \phi_1|$ is minimal. Isotope $E$ in $W \setminus \phi_2$ so as to minimize $|D \cap E|$. Thus, $E \cap D$ consists of circles and arcs. Suppose $\zeta \subset D \cap E$ is a circle of intersection that is innermost in $D$. Let $\Delta_D \subset D$ and $\Delta_E \subset E$ be the discs it bounds. If $\Delta_D \cap \phi_2 = \nil$, we could perform a disc swap on $E$ to reduce $|D \cap E|$ without increasing $|E \cap \phi_1|$; this contradicts the choice of $E$. On the other hand if $\Delta_D \cap \phi_2 \neq \nil$, the intersection consists of a single point. Since $\Delta_E \cap \phi_2 = \nil$, we contradict the fact that $\Delta_D \cup \Delta_E$ bounds a 3-ball in $W$ and the endpoints of $\phi_2$ lie on $S$. Thus, $D \cap E$ has no circles of intersection. 

Similarly, suppose that $\zeta \subset D \cap E$ is an arc of intersection that is outermost in $D$, in the sense that it cobounds an unpunctured subdisc $\Delta \subset D$ with an arc of $\boundary D$ and with interior disjoint from $E$.  Use $\Delta$ to $\boundary$-compress $E$. This converts $E$ into two discs, one of which $E'$ is essential in $W$. Let $E''$ be the one that is inessential in $W$. Note that $|E' \cap \phi_1| \leq |E \cap \phi_1$ and $E' \cap \phi_2 = \nil$. The disc $E''$, together with a disc $F \subset S$, bounds a 3-ball $B \subset W$. If $B \cap \phi_2 = \nil$, then an isotopy of $E$ across $B$ produces a disc contradicting our original choice of $E$. On the other hand, if $B \cap \phi_2 \neq \nil$, observe that $\phi_2 \subset B$. A small isotopy makes $E'$ disjoint from $E$. At which point, the curves $\boundary E' \cup \boundary E$ cobound two annuli in $S$, one of which contains $\boundary \psi$. The union of a push off of the other annulus with the complement in $E'$ of a collar of $\boundary D'$ produces a disc with boundary $b$ contradicting our original choice of $E$. Hence, $a \cap b = \nil$.

Now assume that $a \cap b = \nil$. We will show that $\omega(\phi_1, \phi_2) = 1$. Let $D$ be a disc in $W$ disjoint from $\phi_1$ and with $\boundary D = a$. Let $E\subset W$ be a disc with $\boundary E = b$ and which is disjoint from $\phi_2$.  Isotope $\phi_2$, relative to its endpoints, so as to minimize $|D \cap \phi_2|$. This isotopy can be extended to an isotopy of $E$, supported outside a collar of $\boundary E$, so $E$ remains disjoint from $\phi_2$. In the complement of $\phi_2$, isotope $E$, relative to $\boundary E$, so as to minimize $E \cap D$. We will show that $|D \cap \phi_2| \leq 1$. 

Since $a \cap b = \nil$, the intersection $E \cap D$ consists of circles. Let $\zeta \subset E \cap D$ be a circle that is innermost in $D$; bounding a disc $\Delta_D \subset D$. The circle $\zeta$ also bounds a disc $\Delta_E \subset E$, though it need not be innermost. Then $\Delta_D \cup \Delta_E$ is a sphere bounding a 3-ball $B \subset W$. Since $\Delta_E$ is disjoint from $\phi_2$, we may use a product structure on $B$ between $\Delta_D$ and $\Delta_E$ to isotope $\phi_2 \cap B$ entirely out of $B$. (This isotopy may require $\phi_2$ to be isotoped across $\phi_1$.) This reduces $|D \cap \phi_2|$, a contradiction unless $\phi_2 \cap B = \nil$. However, in that case, we may isotope $\Delta_E$ across $B$, reducing $|D \cap E|$, another contradiction. Thus, $E \cap D = \nil$. 

A $\boundary$-reduction of $W$ using $E$ converts $W$ into a 3-ball $B$ with two discs $E_\pm \subset \boundary W$ that are the scars from $E$. The annulus $A = \boundary B \setminus E_- \cup E_+$ contains $\boundary \phi_2$ as well as the curve $a$. The arc $\phi_2$ is properly embedded in $B$, as is the disc $D$. Since $\phi_2$ is $\boundary$-parallel in $W$ and since $E$ is disjoint from $\phi_2$, the arc $\phi_2$ is $\boundary$-parallel in $B$. Let $\epsilon_2$ be the arc in $\boundary B$ to which $\phi_2$ is parallel. Any two arcs in a twice-punctured sphere with endpoints at the punctures are isotopic in the sphere, so in $\boundary B$, the arc $\epsilon_2$ is isotopic, relative to its endpoints, either to an arc in $A$ disjoint from $a$ or to an arc in $A$ intersecting $a$ exactly once. The former occurs when $a$ does not separates the endpoints of $\epsilon_2$ and the latter when it separates the endpoints of $\epsilon_2$. This isotopy induces an ambient isotopy of $\phi_2$ in $B \subset W$, relative to its endpoints, to an arc that intersects $D$ at most once. 

Thus, $\omega(\phi_1, \phi_2) \leq 1$. However, by Lemma \ref{lem: wrapping well def}, it must equal 1 since $(W, \phi_1)$ and $(W, \phi_2)$ are inequivalent. Thus, (1) is equivalent to (3). 
\end{proof}

The next proposition addresses the proof of Theorem \ref{Main Thm} in the case of low wrapping number.

\begin{corollary}\label{small wrapping}
    Suppose that $(V,\psi)$ is nontrivial and that $(W, \phi_1)$ is a complementary 1-tangle such that $K(\phi_1)$ is the unknot. There is at most one other complementary 1-tangle $(W, \phi_2)$ such that both $K(\phi_2)$ is the unknot and $\omega(\phi_1, \phi_2) = 1$.
\end{corollary}
\begin{proof}
Let $(V, \psi)$ be a nontrivial 1-tangle. Assume that $(W, \phi_1)$, $(W, \phi_2)$, and $(W, \phi_3)$ are complementary 1-tangles such that $K(\phi_i)$ is the unknot for $i \in \{1,2,3\}$. Assume, that $\omega(\phi_1, \phi_2) = 1 = \omega(\phi_1, \phi_3)$.  By Proposition \ref{trivial}, each $(W, \phi_i)$ is trivial.

Let $D \subset W$ be a disc disjoint from $\phi_1$, with $\lambda = \boundary D$  a preferred longitude and with $D \cap \phi_2 = D \cap \phi_3$ a single point. Proposition \ref{reduction} guarantees that there is such a disc. Let $(B, \tau)$ be the $\lambda$-induced 2-tangle from $(V, \psi)$ and let $B'$ be the 3-ball complementary to $B$. Then both $(B', K(\phi_2) \cap B')$ and $(B', K(\phi_3) \cap B')$ are complementary 2-tangles to $(B, \tau)$. Since $(B, \tau)$ is nontrivial by Lemma \ref{lem: nontriv} and has an unknot closure, it must be prime. Thus, by Theorem \ref{2-tangle thm},  $(B', K(\phi_2) \cap B')$ and $(B', K(\phi_3) \cap B')$ are rationally equivalent. Consequently, by Proposition\ref{reduction}, $(W, \phi_2)$ and $(W, \phi_3)$ are equivalent.
\end{proof}

\section{Discs and band sums}\label{band sums}

The basic strategy for proving Theorem \ref{Main Thm} is to rely on results from sutured manifold theory. We will be applying these to the 3-manifold $N$, observing that $\boundary N$ is an orientable genus 2 surface. As is frequently done, certain surfaces in the 3-manifold will used to understand the sutured manifold. We will apply the results in this section to two different closures of a 1-tangle and will study the surfaces corresponding to each. We will have two different preferred longitudes $a, b$ with surfaces $\Delta$ and $Q = \Delta \cap N$ associated to $a$ and surfaces $F'$ and $\Sigma' = F' \cap N$ associated to $b$ respectively. In order to avoid confusion when applying the results of this section, any result that will be applied in both scenarios will be phrased for a preferred longitude $c$ and surfaces $X$ and $R = X \cap N$. An result that is applied in just one of the scenarios will be stated with the appropriate notation.

For the remainder of this section, let $(W, \phi)$ be a trivial complementary 1-tangle and let $a$, $b$, or $c$ denote a preferred longitude bounding a disc in $W$ disjoint from $\phi$. Let $\Delta$, $F'$, or $X$ (respectively) denote an incompressible Seifert surface for $K(\phi)$ transverse to $S$. The intersection of the Seifert surface with $S$ will contain a unique arc which we denote $\epsilon_a$, $\epsilon_b$, or $\epsilon_c$ (respectively). Let $a^*$, $b^*$, or $c^*$ denote a simple closed curve in $S$ bounding a disc in $S$ containing the arc $\epsilon_a$, $\epsilon_b$, or $\epsilon_c$ (respectively) in its interior. Let $Q$, $\Sigma'$, or $R$ denote the intersection of the Seifert surface with $N$. Let $\boundary_{\dag}$ denote the boundary components of $Q$, $\Sigma'$, or $R$ which are parallel in $\circ{S}$ to the curve $\dag \in \{a, b, c, a^*, b^*, c^*\}$ (if any). Note that $\phi$ determines the curve $c$ up to isotopy in $\circ{S}$ but not the arc $\epsilon_c$ or the curve $c^*$. In all cases, we let $\beta$ be the cocore of the 2-handle in $N[a]$, $N[b]$, or $N[c]$ attached along $a$, $b$, or $c$ respectively. We use variations of this notation throughout the rest of the paper.

\begin{definition}
Suppose that $R \subset N$ is a properly embedded surface and that $d \subset \circ{S}$ is a simple closed curve transverse to $\boundary R$. A disc $D \subset N$ is a \defn{$d$-$\boundary$-compressing disc} for $R$ if $D$ is embedded in $N$, $\boundary D$ is the endpoint union of an arc in $R$ and a component of $d\setminus \boundary Q$. 
\end{definition}

\begin{definition}
We say that $X$ is \defn{aligned} with $\phi$ if the following hold:
\begin{enumerate}
\item $X \cap S = \epsilon_c \cup \boundary_c R \cup \boundary_{c^*} R$
\item $X \cap W$ is the union of discs;
\item $R$ is incompressible in $N$.
\end{enumerate}

See Figure \ref{fig:normal} for two examples of aligned Seifert surfaces. When $X$ is aligned with $\phi$, we let $R_{0}$ be the disc in $X \cap W$ with boundary $\phi \cup \epsilon_c$. (Note that $R_0 \subset X \setminus R$.)

A collar neighborhood of $K(\phi)$ in $\boundary X$ is a \defn{standard collar} with respect to $X$ if it intersects each disc of $\nbhd(\boundary \psi) \subset S$ in a single arc each. See Figure \ref{fig:stdcollar}. 

The Seifert surface $X$ is \defn{standard} with respect to $\phi$ if it is aligned with $\phi$, and $X \setminus N[c]$ is a standard collar. Notice that this implies that $\boundary_{c^*} R = \nil$.

The Seifert surface $X$ is \defn{normal} with respect to $\phi$ if it is aligned with $\phi$ and either $X \cap S = \epsilon_c$ or  $R$ does not admit a $d$-$\boundary$-compressing disc for any simple closed curve $d \subset \circ{S}$ which, up to isotopy in $\circ{S}$, intersects $\boundary R$ minimally.
\end{definition}

\begin{figure}[ht!]
\labellist
\small\hair 2pt
\pinlabel {$\phi$} [b] at 222 435
\pinlabel {$b$} [l] at 345 179
\pinlabel {$F$} at 222 377
\pinlabel {$\psi$} [t] at 76 153
\endlabellist
\centering
\includegraphics[scale=0.4]{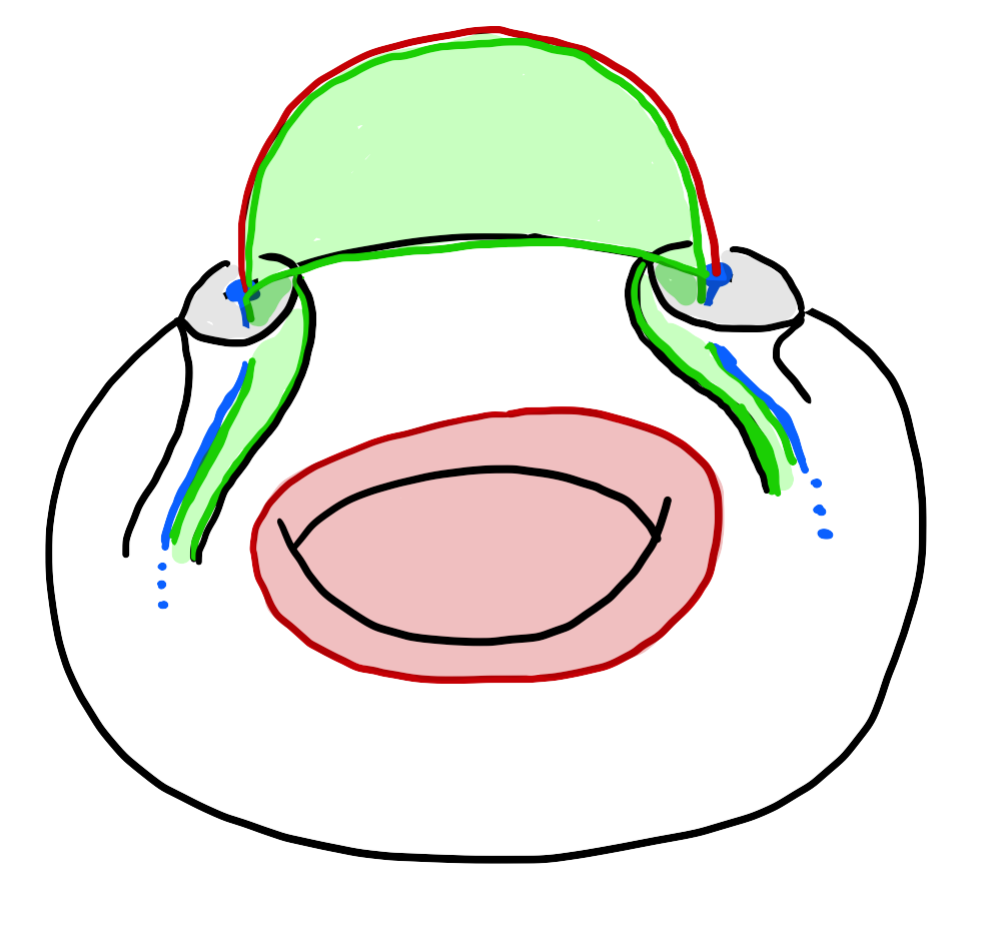}
\caption{A schematic depiction of a standard collar of a Seifert surface for a 1-tangle closure $K(\phi)$. The collar is shown in green. }
\label{fig:stdcollar}
\end{figure}

\begin{figure}[ht!]
\centering
\includegraphics[scale=0.6]{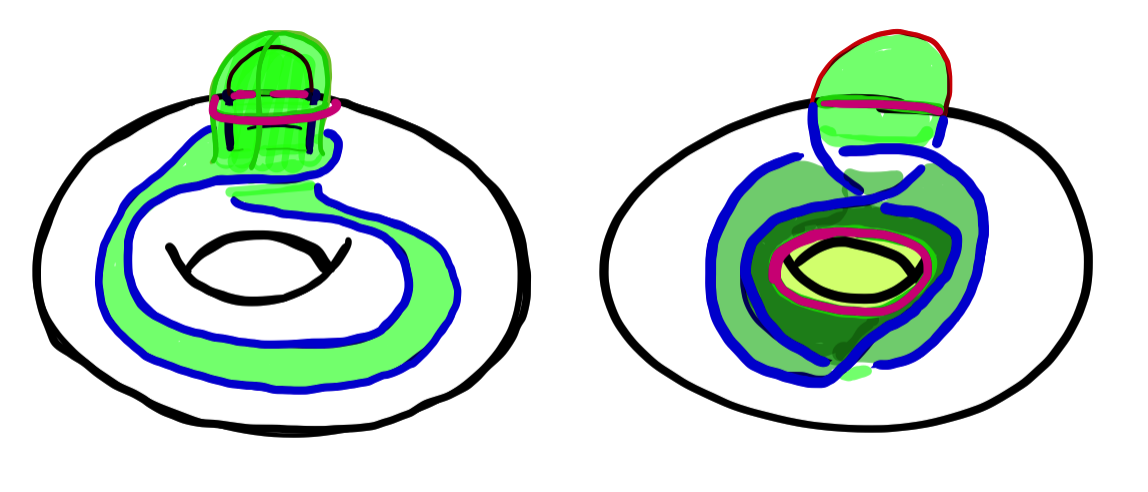}
\caption{On the left is an example of an aligned disc Seifert surface. Notice $|\boundary_{c^*} R| = 1$ and $\boundary_c R = \nil$. On the right is an example of an aligned genus 1 Seifert surface with $|\boundary_c R| = 2$ and $\boundary_{c^*} R = \nil.$ The Seifert surface on the left is a disc obtained by starting with a flat elliptical disc and then pushing one end around the solid torus $V$ and through the disc, creating a bulge. The Seifert surface on the right is the one obtained from Seifert's algorithm applied to the given diagram of the Figure 8 knot.}
\label{fig:normal}
\end{figure}

When $(W, \phi)$ is trivial, it can be helpful to consider the 3-manifold $N$ from the perspective of $K(\phi)$. If $c \subset \circ{S}$ is a preferred longitude bounding a disc in $W$ disjoint from $\phi$, then the 3-manifold $N[c]$ is the closure of the complement of $\nbhd(K(\phi))$. If $\beta$ is the cocore of the 2-handle in $N[c]$ attached along $c$, the 3-manifold $N$ is then the exterior of the arc $\beta$. If $X$ is standard with respect to $\phi$, then $X\setminus N[c]$ is a standard collar with respect to $\phi$ and $R[c]$ is transverse to $\beta$. Conversely, if $X$ is any Seifert surface for $K(\phi)$ such that $X \setminus N[c]$ is a standard collar,  and such that $X$ is transverse to $\beta$, then $X$ is standard with respect to $\phi$ if and only if $R[c] \setminus \beta$ is incompressible in $N = N[c]\setminus \beta$. 

\begin{lemma}\label{standard position}
Assume that $(V, \psi)$ is a nontrivial 1-tangle without a local knot and that $(W, \phi)$ is a trivial complementary 1-tangle. Suppose that $X$ is incompressible. If $X\setminus N[c]$ is a standard collar and $X$ has been isotoped, via an isotopy supported in a small regular neighborhood of $N[c]$, so as to intersect $\beta$ minimally, then $X$ is standard with respect to $\phi$.
\end{lemma}
\begin{proof}
As we discussed, we may consider $N[c]$ as the exterior of a regular neighborhood of $K(\phi) = \boundary X$. Hence, we may isotope $X$, relative $\boundary X$ so that $X \setminus N[c]$ is a collar neighborhood of $\boundary X$, which we may assume is disjoint from the endpoints of $\beta$. A small isotopy also ensures that $X$ is transverse to $\beta$. If $R$ is compressible, then since $X$ is incompressible and $N[c]$ is irreducible, we may isotope $X$ relative to $(X \setminus N[c])$ to reduce $|X \cap \beta|$. The lemma follows.
\end{proof}

To achieve normality, we have to work harder.

\begin{definition}\label{curve types}
See Figure \ref{fig:compressiontypes}. Suppose that $D$ is a $\boundary$-compressing disc for $Q$ such that $\boundary D \cap \boundary N$ is an arc $\delta \subset S$. For $x,y \in \{a, a^*, \epsilon_a\}$, we say that $D$ is an $(xy)$-compression if it joins the corresponding components of $\boundary Q \cap S$. 
\end{definition}

\begin{figure}[ht!]
\centering
\includegraphics[scale=0.4]{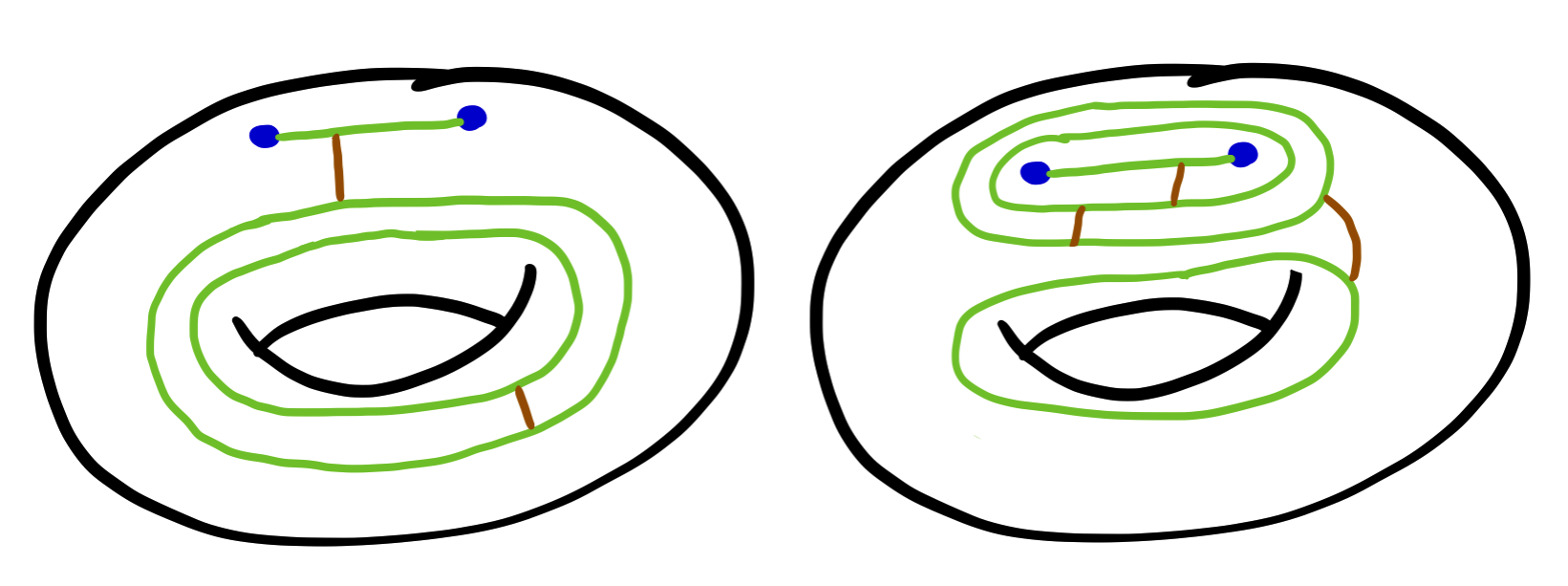}
\caption{We show the surface $S$ with the two points of $\boundary \psi$. The green curves represent the arc $\epsilon_a$ and the components of $\boundary_a Q$ and $\boundary_{a^*} Q$. The brown arcs are possibilities for the arc $\delta$. On the left we show a $(\epsilon_a a)$-compression and a $(aa)$ compression. In the middle, we show a $(\epsilon_a a^*)$-compression, a $(a^* a^*)$-compression, and a $(a a^*)$-compression. }
\label{fig:compressiontypes}
\end{figure}

\begin{lemma}\label{normal position}
Suppose that $\Delta$ is incompressible. Then there is an isotopy of $\Delta$, relative to $K(\phi)$, such that after the isotopy $\Delta$ is normal with respect to $\phi$. If $\Delta$ was originally aligned with $\phi$, then this isotopy is supported in a small regular neighborhood of $N[a]$.
\end{lemma}
\begin{proof}
By Lemma \ref{standard position}, we may isotope $\Delta$, relative to $K(\phi_1)$ so that it is aligned with $\phi_1$, if it was not already.  It may  be that $\boundary_{a^*} Q \neq \nil$. If $\boundary Q = \epsilon_a$, then the result is immediate. Suppose, therefore, that $\boundary_a Q \cup \boundary_{a^*} Q \neq \nil$. Assume $Q$ admits a $d$-$\boundary$-compressing disc $D$ for some simple closed curve $d \subset \circ{S}$ intersecting $\boundary Q$ minimally up to isotopy in $\circ{S}$. Let $\delta = \boundary D \cap S \subset d$.

We will examine the possibilities for the locations of $\boundary \delta$; in each case the disc $D$ will guide an isotopy of $\Delta$ preserving the fact that $\Delta$ is aligned with $\phi$ and reducing $|\boundary Q|$.

Each component of $S \setminus (\boundary_a Q \cup \boundary_{a^*} Q)$ is one of the following:
\begin{enumerate}
    \item[(i)] an annulus disjoint from $\epsilon_a$
     \item[(ii)] an annulus containing $\epsilon_a$
      \item[(iii)] a pair of pants disjoint from $\epsilon_a$ with two components in $\boundary_a Q$ and one in $\boundary_{a^*} Q$,
    \item[(iv)] a disc containing $\epsilon_a$
\end{enumerate}

Let $S_0 \subset S \setminus (\boundary_a Q \cup \boundary_{a^*} Q$ be the component containing $\delta$.

If $S_0$ is of type (i), then $Q$ would be compressible, a contradiction. 

If $S_0$ is of type (ii), then $D$ is $(\epsilon_a \epsilon_a)$-compression, a $(\epsilon_a a)$-compression, or an $(aa)$-compression. If it is a $(\epsilon_a \epsilon_a)$-compression, by the minimality of $|d \cap \boundary Q|$, the arc $\delta$, together with a subarc of $\epsilon_a$ forms a curve parallel in $\circ{S}$ to $a$. Use $D$ to isotope $\Delta$ so as to perform the $\boundary$-compression of $Q$. This converts the disc $\Delta_0$ into an annulus with one component $a'$ parallel to $a$. Since $a$ bounds a disc in $W$ disjoint from $\phi$ and since $\Delta$ is incompressible, $a$ bounds a disc $\Delta' \subset \Delta$. If $\Delta'$ had interior disjoint from $S$ then the arc $\boundary D \cap \Delta$ would have been inessential in $\Delta$. Thus, the interior of $\Delta'$ intersects $S$. The curve $a'$ bounds a disc in $W$ with interior disjoint from $\Delta$ and $N[a]$ is irreducible, so we may isotope $\Delta'$ to lie in $W$ and then isotope a bit further to eliminate the curve $a'$ from $\Delta \cap S$. Thus, we have an isotopy of $\Delta$, supported in a small neighborhood of $N[c]$ which strictly reduces $|\Delta \cap S|$ and preserves the fact that it is aligned with $\phi$.

Similarly, if $D$ is a $(\epsilon_a a)$-compression, we may use it to isotope $\Delta$, via an isotopy transverse to $\boundary N[a]$ to eliminate a point of $\Delta \cap \beta$ by sliding $\boundary \Delta$ across an endpoint of $\beta$. In this case, as in the previous one, the arc $\epsilon_a$ has changed. If $D$ is an $(aa)$-compression, it either joins a component of $\boundary_{a} Q$ to itself or joins distinct components of $\boundary_a Q$. In the first case, the isotopy guided by $D$ converts a component of $\boundary_a Q$ into a component of $\boundary_a Q$ and a component of $\boundary_{a^*} Q$ and the component of $\Delta \cap W$ containing $\boundary \delta$ into an annulus $A$. The components of $\boundary A$ each bound discs in $W\setminus \phi$, so as in the case of the $(\epsilon_a \epsilon_a)$-compression, we may isotope $\Delta$, via an isotopy supported in a small neighborhood of $N[c]$ so as to strictly reduce $|\Delta \cap S|$ and preserve the fact that $\Delta$ is aligned with $\phi$.  In the second case, the isotopy from $D$ converts two components of $\boundary_a Q$ into a component of $\boundary_{a^*} Q$ and in $W$ it bands to discs together into a new disc. We again have the isotopy we are looking for. 

If $S_0$ is of type (iii), then $D$ is either an $(aa)$, $(aa^*)$, or $(a^* a^*)$-compression. Each case is much like the cases we have already considered and we again find the isotopy we were seeking. 

Finally, if $S_0$ is of Type (iv), $D$ must be an $(\epsilon_a a^*)$-compression. An isotopy of $\Delta$ using $D$ eliminates a component of $\boundary_{a^*} Q$ and changes the arc $\epsilon_a$. In $W$, it bands two discs together and we again find the isotopy we were seeking. 
\end{proof}

We now present two additional observations one concerning normal surfaces and the other concerning standard surfaces.

\begin{definition}
Suppose that $L \subset S^3$ is a 2-component split link and that $B = I \times I$ is embedded in $S^3$ such that one component of $(\boundary I )\times I$ is contained in each component of $L$. The knot $K = (L \setminus (\boundary I \times I)) \cup (I \times \boundary I)$ is obtained by a \defn{band sum} of $L$ using the band $B$ to the link $L$. The band sum is \defn{trivial} if there exists a 2-sphere in $S^3$ whose intersection with $B$ is a single arc joining the components of $I \times \boundary I$.
\end{definition}

\begin{lemma}\label{bubbles}
Suppose that $(V,\psi)$ is a nontrivial 1-tangle and that $(W, \phi_1)$ and $(W,\phi'_2)$ are complementary 1-tangles, with $K(\phi_1)$ an unknot. Let  $\Delta$ be a disc with boundary $K(\phi_1)$ which is  normal with respect to $\phi_1$. Let $a$ be a preferred longitude bounding a disc in $W$ disjoint from $\phi_1$,  and let $Q = \Delta \cap N$. Assume that $(W, \phi'_2)$ is not equivalent to $(W, \phi)$, that it does not contain a local knot, and that $\phi'_2 \cap (\Delta_0)= \boundary \psi$. If $\boundary_a Q = \nil$, then $K(\phi'_2)$ is a nontrivial band sum of $\kappa$ and $K(\phi_1)$.
\end{lemma}
\begin{proof}
Each component $D$ of $\Delta \cap W$ other than $\Delta_0$ is a disc whose boundary bounds a disc $E \subset S$ containing the arc $\epsilon_a$.  The sphere $D \cup E$ bounds a 3-ball in $W$ containing $\phi$.  Let $\Delta'$ be the frontier of a regular neighborhood of $\Delta \cap W$. Observe that $\Delta'$ is the disjoint union of $2|\boundary_{a^*} Q| + 1$ discs. Let $D$ be the outermost disc and $E \subset S$ the disc containing $\epsilon_a$ with $\boundary D = \boundary E$. There is a product structure $D \times I$ on the 3-ball $B' \subset W$ with boundary $D \cup E$ such that $D = D \times \{1\}$ and $E = D \times \{0\}$. We may arrange that the disc $\Delta_0$ is vertical in this product structure; specifically $\Delta_0 = \epsilon_a \times [0, 1/2]$. We may also arrange that each disc of $(\Delta \cap W) \setminus (D \cup \Delta_0)$ consists of a vertical annulus and a horizontal disc. That is, the annulus is of the form $a' \times [0, x]$ for some $x \in (1/2, 1)$ and simple closed curve $a' \subset E$ and the disc is of the form $E' \times \{x\}$ for some subdisc $E' \subset E$.  We may then isotope $\phi'_2$ so that $\phi'_2 \cap B' = \boundary \psi \times I$. Push the subarcs $ \boundary \psi \times [0,1/2]$ slightly off $\Delta_0$, as in Figure \ref{band sum fig}. 

\begin{figure}[ht!]
\labellist
\small\hair 2pt
\pinlabel {$B'$} at 310 346
\pinlabel {$\phi'_2$} [b] at 292 439
\pinlabel {$\Phi$} at 411 347
\pinlabel {$\Delta_0$} at 410 185
\pinlabel {$S$} at 73 97
\pinlabel {$W$} at 161 291
\pinlabel {$D$} [r] at 259 363
\pinlabel {$E$} [b] at 244 168
\pinlabel {$\psi$} [t] at 310 47
\pinlabel {$V$} at 70 47
\endlabellist
\centering
\includegraphics[scale=0.5]{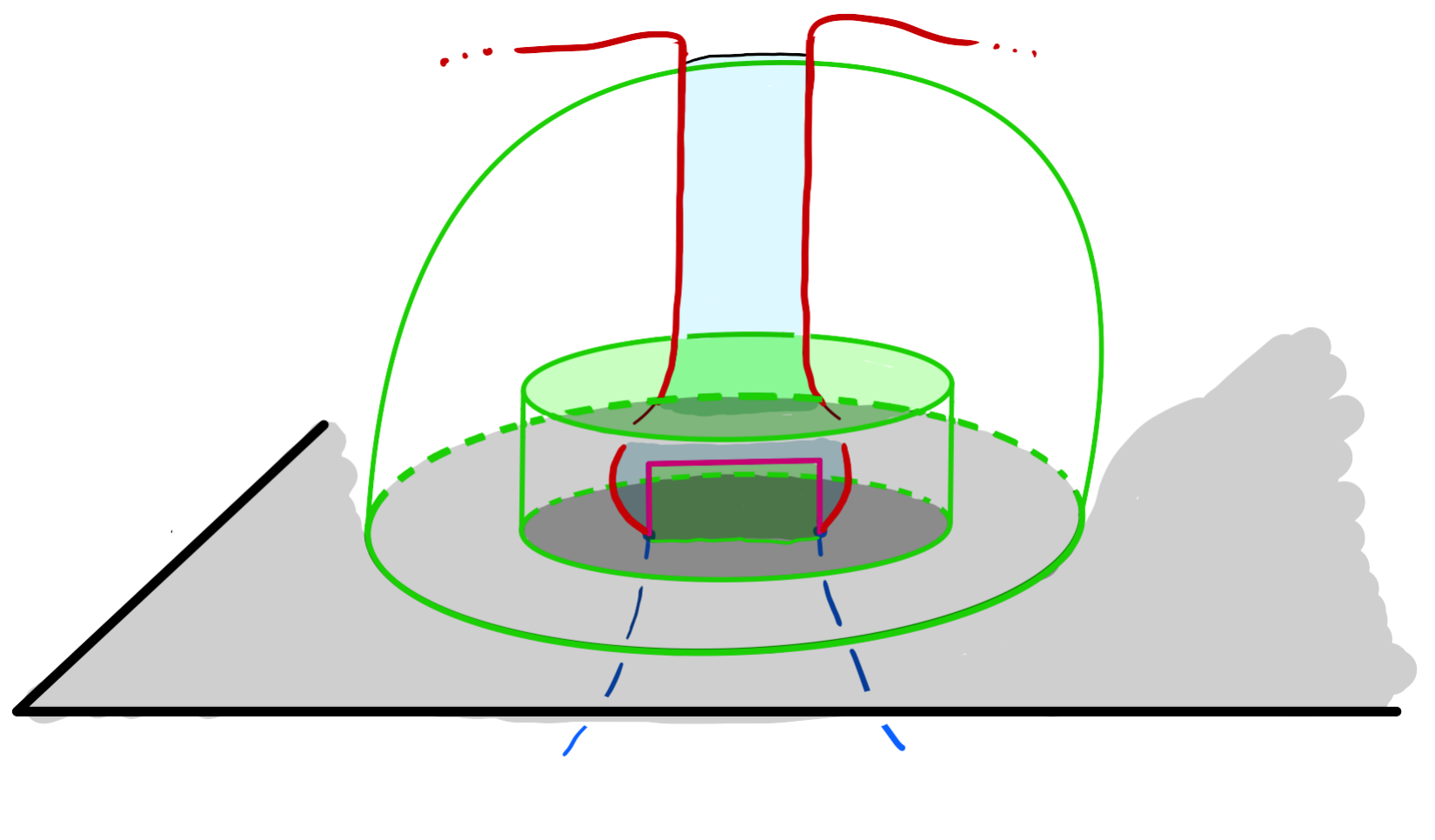}
\caption{ The knot $K(\phi'_2)$ can be by banding the  knot $K(\phi_1)$ to a knot isotopic in $W$ to $\phi_1 \cup \phi'_2$. The band $\Phi$ is shown in blue. The green discs are $\Delta \cap W$. See the proof of Lemma \ref{bubbles}.}
\label{band sum fig}
\end{figure}
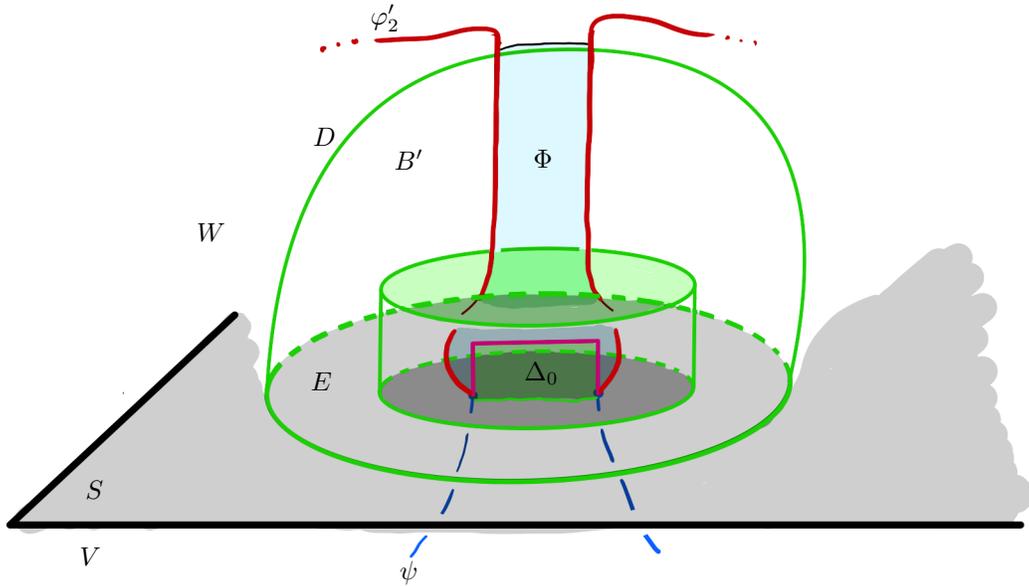

Let $\Phi$ be the band that is the union of $\epsilon_a \times [1/2,1]$ with the narrow discs created by isotoping $\phi'_2$ off $\Delta_0$. $\Phi$ is a band joining the knot $\kappa = (\phi'_2 \setminus B') \cup (\epsilon_a \times \{1\})$ to $K(\phi_1)$, as in Figure \ref{band sum fig}. Observe that in $W$, $\kappa$ is isotopic to $\phi_1 \cup \phi'_2$. The knot $K(\phi'_2)$ is the result of attaching the band $\Phi$ to $\kappa \cup K(\phi_1)$. 

We claim that since $K(\phi_1)$ is the unknot, then this is a nontrivial band sum. Let $P$ be the sphere that is the frontier of a small regular neighborhood of $\Delta$. Note that $P \cap W$ consists of $2|\boundary_{a^*} R| + 1$ discs. For convenience, we may slightly expand $B'$ and adjust the product structure so that each of these discs are interior to $B'$ and are the union of a vertical annulus and a horizontal disc. The sphere $P$ separates  $\kappa$ from $K(\phi_1)$, so $K(\phi'_2)$ is the band sum of $\kappa \cup K(\phi_1)$ using $\Phi$. 

Suppose, for a contradiction, that the band sum is trivial. Then there is a sphere $P$ separating $\kappa$ from $K(\phi_1)$ and intersecting $\Phi$ once. We may choose $P$ to minimize $|P \cap S|$. Note that $P \cap W$ consists of discs, each of which has boundary parallel to $a$ or $a^*$. Each component of $P \cap W$ of the latter kind intersects the band $\Phi$ in a single arc, so there is at most one component of the latter kind. If no component of $P \cap S$ is parallel to $a$, then $P \cap V$ is a single disc, contradicting Lemma \ref{Jaco lem} or the minimality of $|P \cap S|$. 

On the other hand, if some component $D \subset P \cap W$ has boundary parallel to $S$, then $\omega(\phi_1, \phi'_2) = 0$ and the result follows from  Lemma \ref{lem: wrapping well def}. 
\end{proof}

\begin{lemma}\label{disjoint arc}
Let $(V,\psi)$ be a nontrivial 1-tangle and $(W,\phi_2)$ a trivial complementary 1-tangle. Suppose that $F'$ is an incompressible Seifert surface for $K(\phi_2)$ that is standard with respect to $\phi_2$. Let $\Sigma' = F \cap N$.   Let $b$ be a preferred longitude bounding a disc in $W$ disjoint from $\phi_2$. If $\Sigma'[b]$ is properly isotopic in $N[b]$ to a surface disjoint from the cocore $\beta$, then either $F'$ is itself disjoint from $\beta$ or $\Sigma'$ is $d$-$\boundary$-compressible in $N$ for some curve $d \subset \boundary N[b]$ intersecting $\boundary \Sigma'$ minimally up to isotopy. 
\end{lemma}
\begin{proof}
Since $F'$ is standard with respect to $\phi_2$, $\boundary_{b^*} \Sigma' = \nil$. Hence, the surface $\Sigma'[b]$ is properly embedded in $N[b]$ and has a single boundary component. Indeed, $\Sigma'[b]$ is obtained by removing a collar neighborhood of $\boundary F'$ from $F'$. Since $\Sigma'[b]$ is isotopic in $N[b]$ to a surface disjoint from $\beta$, the arc $\beta$ is properly isotopic in $N[b]$ to an arc $\beta'$ disjoint from $\Sigma'[b]$.

Let 
\[
H\co I \times I \to N[b]
\]
be such an isotopy. In particular, $H(t, 0)$ is a parameterization of $\beta$ and $H(t, 1)$ is a parameterization of $\beta'$ and for all $t \in [0,1]$, $H(0, t)$ and $H(1,t)$ lie in the torus $\boundary N[b]$. We may assume that $H$ is transverse to $\Sigma'[b]$, so that $H^{-1}(\Sigma'[b])$ is a 1-manifold $G$: the union of embedded arcs and circles in $I \times I$. Assume that we have chosen $H$ and $\beta'$ so as to minimize $|G|$.

Each circle of $G$ bounds a disc in $I \times I$. Since $\Sigma'[b]$ is incompressible, we may surger $H$ (without affecting $\beta$ or $\beta'$) to eliminate all circles from $G$, contradicting our choice of $H$. Consequently, each component of $G$ is an arc. Since $|\beta' \cap \Sigma'[b]| = 0$, each arc of $G$ is one of the following:
\begin{enumerate}
\item a \defn{spanning arc}: an arc having one endpoint in $\{0\} \times I$ and the other in $\{1\} \times I$;
\item an arc of $G$ having both endpoints in $(0,1) \times \{0\}$;
\item an arc of $G$ having an endpoint in $(0,1) \times \{0\}$ and its other endpoint in $(\boundary I) \times I$;
\item an arc having both endpoints in $\{0\} \times I$ or both endpoints in $\{1\} \times I$.
\end{enumerate}

If there is a spanning arc, we may replace $\beta'$ with $H(I \times \{t_0\})$ for an appropriate value of $t_0$ and again contradict the choice of $H$ and $\beta'$. Hence, $G$ contains no spanning arcs. Let $\zeta \subset G$ be an outermost arc, cobounding a subdisc $\Delta' \subset I \times I$ with (respectively):
\begin{enumerate}
\item[(2)] a subarc of $(0,1) \times \{0\}$, 
\item[(3)] a subarc of $(\boundary I) \times I$ and a subarc of $(0,1) \times \{0\}$
\item[(4)] a subarc of $(\boundary I) \times I$.
\end{enumerate}
Standard 3-manifold theory (eg. a doubling argument combined with an application of Dehn's lemma and the loop theorem) then implies that $\Sigma$ is $\boundary$-compressible via a $\boundary$-compressing disc $D$ which is the intersection with $N$ of a disc $D$ such that $\boundary D$ is the union of an arc in $\Sigma$ with either (2') a subarc of $\beta$, (3') a subarc of $\beta$ and an arc in $\boundary N[c]$, or (4') just an arc in $\boundary N[c]$, corresponding to (2), (3), and (4) respectively. Since $\Sigma'$ is incompressible, (2') does not occur. If (3') occurs, then the frontier of a regular neighborhood of $D$ is a disc satisfying (4'). Conclusion (4') is what the lemma claims.
\end{proof}

\section{Proof of Theorem \ref{Main Thm}}\label{final proof}

We restate the theorem for convenience:

\begin{maintheorem}
If $(V, \psi)$ is nontrivial, then for any arc $\phi \subset W$ with $K(\phi)$ the unknot, the arc $\phi \subset W$ is trivial and, up to isotopy in $W$ relative to $\boundary \psi$, there are at most two such arcs $\phi$. Furthermore, if $\phi_1, \phi_2$ are two such arcs then $\omega(\phi_1, \phi_2) = 1$. 
\end{maintheorem}

For the proof, we make use of the following theorem of the author, which as promised, relies on sutured manifold theory terminology. We give a simplified (though logically complete) statement, adapted to our context. It is obtained from the statement as given in \cite{Taylor-2h}*{Theorem 2.1} by strengthening the hypotheses and weakening the conclusion.

\begin{theorem}[{cf. \cite{Taylor-bandtaut}*{Theorem 10.7}}]\label{Taylor thm}
    Suppose that $(N,\gamma)$ is a taut sutured manifold with $\boundary N$ a genus 2 surface and with three nonseparating curves $\gamma \subset \boundary N$, dividing $\boundary N$ into two pairs-of-pants. Let $b \subset \gamma$ be one of the curves. Suppose that $Q \subset N$ is an incompressible and $b$-$\boundary$-incompressible connected surface with $\boundary Q$ intersecting $\gamma$ minimally and nontrivially. Then one of the following is true:
    \begin{enumerate}
        \item[(SM1)] $N[b]$ has a summand that is a nontrivial lens space;
        \item[(SM2)] The arc $\beta$, which is the cocore of the 2-handle in $N[b]$ attached along $b$, can be properly isotoped to be disjoint from the first surface $\Sigma[b]$ in a taut sutured manifold hierarchy for $(N[b], \gamma \setminus b)$. Furthermore, $(\Sigma[b],\boundary \Sigma[b])$ can be taken to represent $\pm y$ for any nontrivial $y \in H_2(N[b], \boundary N[b])$;
        \item[(SM3)] $-2\chi(Q) + |\boundary Q \cap \gamma| \geq 2|\boundary Q \cap b|$
    \end{enumerate}
\end{theorem}

We prove the following and then conclude the proof of Theorem \ref{Main Thm}. 
\begin{proposition}\label{Seifert business}
Suppose that $(V, \psi)$ is a nontrivial 1-tangle and that $(W, \phi_1)$ and $(W, \phi_2)$ are trivial complementary 1-tangles, with $\omega(\phi_1, \phi_2) \geq 2$. Let $b \subset \circ {S}$ be a preferred longitude bounding a disc in $W$ disjoint from $\phi_2$. Suppose that $K(\phi_1)$ is the unknot. Then $K(\phi_2)$ is not the unknot and one of the following holds:
\begin{enumerate}
\item There is a complementary 1-tangle $(W, \phi'_2)$ equivalent to $(W, \phi_2)$ such that $K(\phi_2)$ is the result of a nontrivial band sum of $K(\phi_1)$ and $\kappa = \phi_1 \cup \phi'_2$. 
\item $x(N) \leq x(N[b])$. In particular there exists a minimal genus Seifert surface $F$ for $K(\phi_2)$ which is standard with respect to $\phi_2$ such that $\Sigma[b] = F \cap N[b]$ is isotopic in $N[b]$ to a surface disjoint from $\beta$.
\end{enumerate}
\end{proposition}

\begin{proof}[Proof of Proposition \ref{Seifert business}]
Assume that $(V,\psi)$ is a nontrivial 1-tangle and that $(W, \phi_1)$ and $(W, \phi_2)$ are trivial complementary 1-tangles such that $K(\phi_1)$ is the unknot and $\omega(\phi_1, \phi_2) \geq 2$. Let $a, b \subset \circ{S}$ be preferred longitudes bounding discs in $W$ disjoint from $\phi_1, \phi_2$ respectively.  Let $\beta \subset N[b]$ be the cocore of the 2-handle attached along $b$.

The 3-manifold $N = V \setminus \nbhd(\psi)$ has a single boundary component of genus 2. Let $\mu \subset \boundary N$ be a meridian of $\psi$. We will eventually define a curve $b' \subset \boundary N$ so that $(N,\gamma) = (N, \mu \cup b \cup b')$ is a sutured manifold to which we can profitably apply Theorem \ref{Taylor thm}.

Let $\Delta$ be a disc with $\boundary \Delta = K(\phi_1)$. By Lemma \ref{normal position}, it can be isotoped so that it is normal with respect to $\phi_1$. Let $Q = \Delta \cap N$. Recall that $\boundary Q \cap S = \epsilon_a \cup \boundary_a Q \cup \boundary_{a^*} Q$. Let $\boundary_0 Q = \boundary \Delta$. We may assume that $a$, $a^*$, $b$, $\epsilon_a$, $\boundary_a Q$, and $\boundary_{a^*} Q$ all intersect minimally up to isotopy in $\circ{S}$. 

By Lemma \ref{standard position}, there is a minimal genus Seifert surface for $K(\phi_2)$ which is standard with respect to $\phi_2$. Give it an orientation and let $y \in H_2(N[b], \boundary N[b])$ be the class represented by its intersection with $N[b]$. 

\textbf{Case 1:} $\boundary_a Q = \nil$.

Let $\Delta_0 \subset \Delta \cap W$ be the disc with boundary $\phi_1 \cup \epsilon_a$. Isotope $\phi_2$ in $W$ relative to $\boundary \phi_2$ to an arc $\phi'_2 \subset W$ with interior disjoint from $\Delta_0$. This may require isotoping it across $\phi_1$. Let $\kappa = \phi_1 \cup \phi'_2$. Since $K(\phi_1)$ is the unknot, according to Lemma \ref{bubbles}, $K(\phi_2)$ is the nontrivial band sum of $\kappa$ with $K(\phi_1)$. By a theorem of Scharlemann \cite{Scharlemann-bandsum}, $K(\phi_2)$ is not the unknot.

  \textbf{Case 2:} $\boundary_a Q \neq \nil$.

By Lemma \ref{wrapping 1 lem}, $a \cap b \neq \nil$. We claim also that neither $\epsilon_a$ nor $a^*$ are disjoint from $b$. To see this, as in Figure \ref{Fig:AnnulusArg}, consider the annulus $A = S \setminus \nbhd(b)$. Let $b_\pm$ be the components of $\boundary A$. The intersection $a \cap A$ consists of spanning arcs and returning arcs and there are the same number of returning arcs incident to each of $b_\pm$. Since $a$ and $b$ are isotopic in $S$ (but not in $\circ{S}$), their algebraic intersection number is zero, so there is at least one returning arc incident to each of $b_\pm$. Since $a$ and $b$ intersect minimally in $\circ{S}$, each returning arc cobounds a disc in $A$ with a subarc of $\boundary A$ containing at least one point of $\boundary \psi$. If such a disc were to contain both points, we could not have returning arcs incident to both of $b_\pm$. Thus, each such disc contains exactly one component of $\boundary \psi$. In particular, the returning arcs incident to each of $b_\pm$ are all parallel. Let $D_\pm$ be the innermost discs in $A$ cut off by returning arcs incident to each component of $\boundary A$. The arc $\epsilon_a$ is disjoint from $a$ and $\boundary \epsilon_a = \boundary \psi$. Thus, each of $D_\pm$ must contain an arc of $\epsilon_a \cap A$, so $\epsilon_a \cap b \neq \nil$. Since the curve $a^*$ is the frontier of a regular neighborhood of $\epsilon_a$, the curve $a^*$ must also intersect $b$.

\begin{figure}[ht!]
\labellist
\small\hair 2pt
\pinlabel {$b_+$} [br] at 553 1680
\pinlabel {$b_-$} [tl] at 714 977
\pinlabel {$\epsilon_b$} [b] at 913 1475
\pinlabel {$D_+$} at 273 1279
\pinlabel {$D_-$} [b] at 1300 980
\endlabellist
\centering
\includegraphics[scale=0.1]{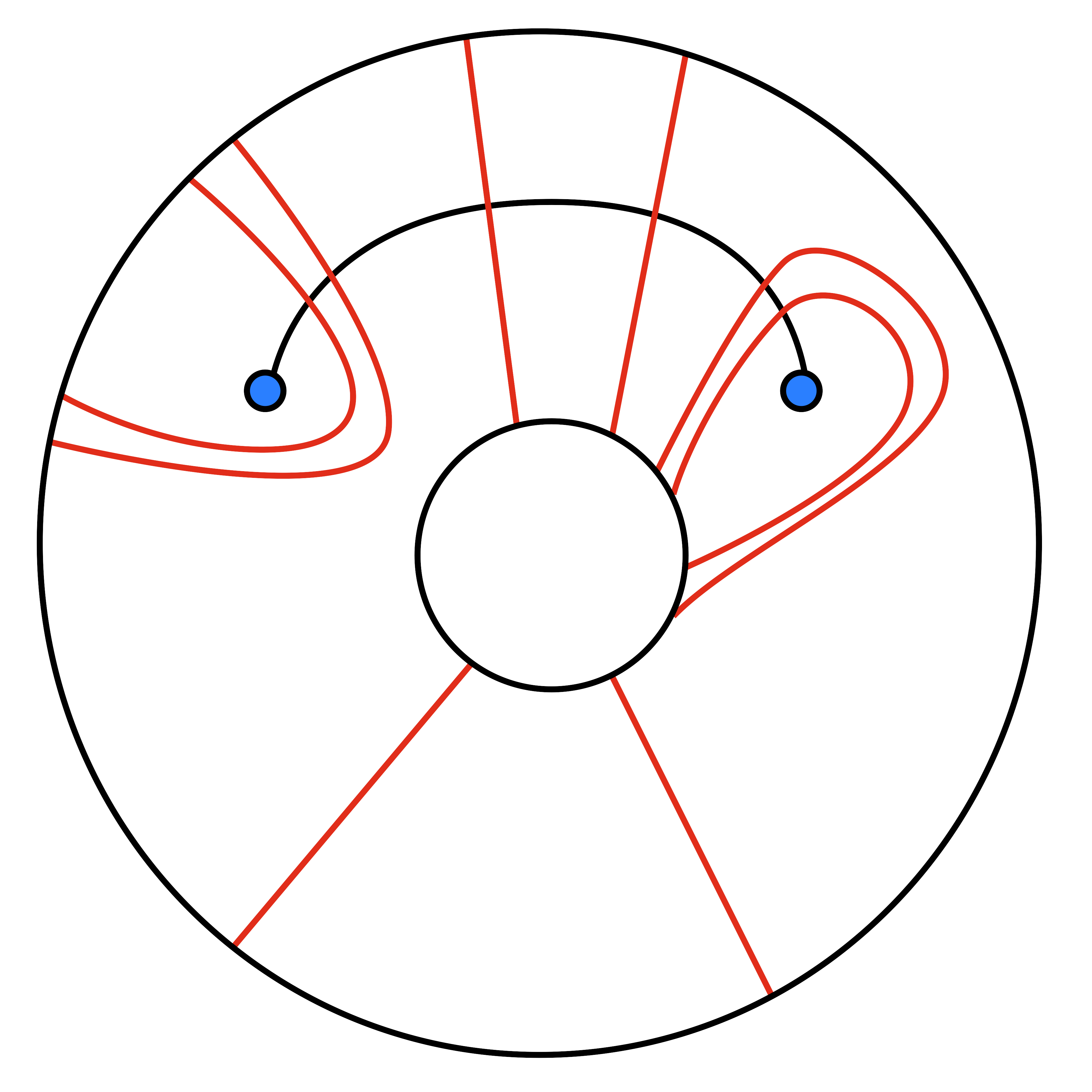}
\caption{The annulus $A = S \setminus \nbhd(b)$ contains the arc $\epsilon_b$. The curve $a$ (shown in red) intersects $A$ in spanning arcs and returning arcs. There must be the same number of returning arcs incident to each boundary component of $A$. Although in our diagram, each arc of $a \cap A$ intersects $\epsilon_b$ exactly once, in general the arcs may intersect $\epsilon_b$ more times.}
\label{Fig:AnnulusArg}
\end{figure}
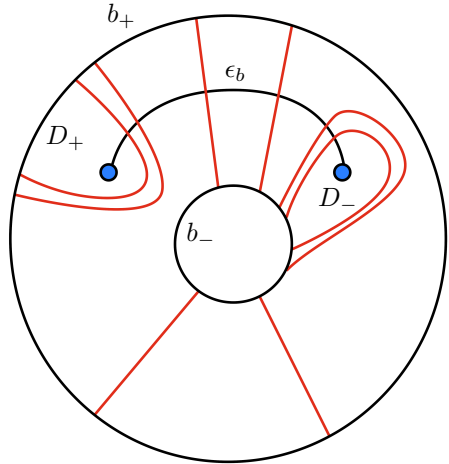

We now define the curve $b' \subset \circ{S}$. Although there are many ways to do this, we will need to be careful about intersection numbers. See Figure \ref{Fig:AnnulusArg2}.

\begin{figure}[ht!]
\labellist
\small\hair 2pt
\pinlabel {$b_+$} [br] at 577 1710
\pinlabel {$b_-$} [tl] at 802 1071
\pinlabel {$b'$} [bl] at 237 567
\pinlabel {$D$} at 323 1280
\endlabellist
\centering
\includegraphics[scale=0.1]{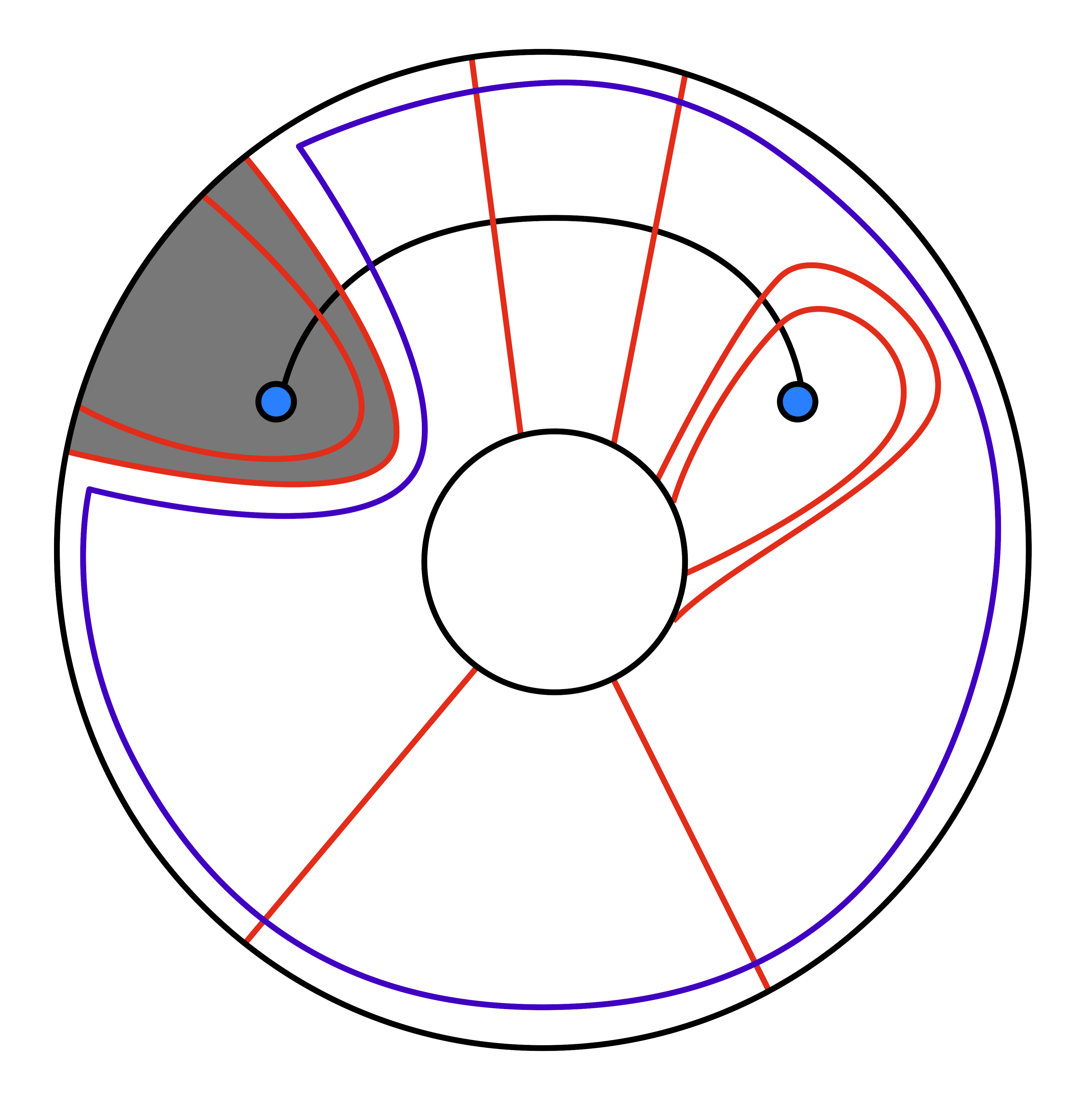}
\caption{We modify Figure \ref{Fig:AnnulusArg} to show the disc $D$ and the curve $b'$.}
\label{Fig:AnnulusArg2}
\end{figure}
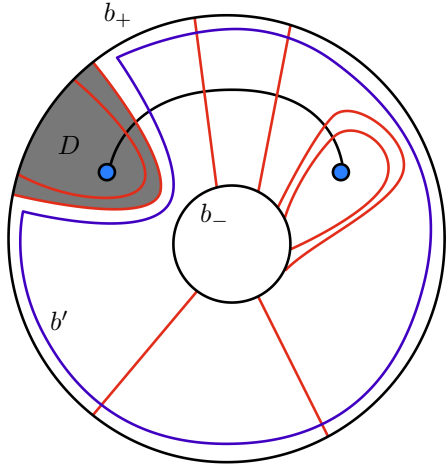

Let $\zeta \subset a \cap A$ be an outermost returning arc, incident to a component $b_+ \subset \boundary A$, and let $D \subset A$ be the disc it cobounds with a subarc of $b_+$. Let $b' \subset (A \setminus \boundary \psi)$ be the simple closed curve that is the frontier of small regular neighborhood of $b_+ \cup D$.  Note that $b'$ is disjoint from $b$, parallel in $S$ to $b$, and that $b \cup b'$ separate the points of $\boundary \psi$. We also note that, by the construction, we have the following inequalities:
    \begin{enumerate}
        \item[(I1)] $|\boundary_a Q \cap b'| \leq |\boundary_a Q \cap b| - 2$
        \item[(I2)] $|\boundary_{a^*} Q \cap b'| \leq |\boundary_{a^*} Q \cap b| - 2$ (assuming $\boundary_{a^*} Q \neq \nil$.)
        \item[(I3)] $|\epsilon_a \cap b'| \leq |\epsilon_a \cap b| - 1$
\end{enumerate}

The surface $\boundary N \setminus \gamma$ is the union of two pairs of pants, each incident to each component $\gamma =  \mu \cup b \cup b'$, so $(N, \gamma)$ is a sutured manifold. Since no component of $\gamma$ can bound a disc in $N$, and since $N$ is irreducible, the sutured manifold $(N, \gamma)$ is taut. Since $a \cap b \neq \nil$, and since $Q$ is incompressible and $b$-$\boundary$-incompressible, we may apply Theorem \ref{Taylor thm}. Since $N[b] \subset S^3$ is a solid torus,  it does not have a lens space summand. Thus, either Conclusion (SM2) or Conclusion (SM3) holds. 

\textbf{Case 2a: } Conclusion (SM2) holds.

According to Conclusion (SM2), $(N[b], \gamma \setminus b)$ is taut and the arc $\beta$ may be properly isotoped  in $N[b]$ to an arc disjoint from the first surface $\Sigma[b]$ in a taut sutured manifold hierarchy for $(N[b], \gamma \setminus b)$. Before making use of the isotopy of $\beta$, we consider the surface $\Sigma[b]$ in more detail.  According to Conclusion (SM2), $(\Sigma[b],\boundary \Sigma[b])$ may be chosen to represent $\pm y$. The surface $\Sigma[b]$ is taut and conditioned \cite{Scharlemann}*{Definition 2.4} in $(N[b], \gamma \setminus b)$. ``Conditioned'', together with the definition of $y$ and the fact that  that $(\Sigma[b],\boundary \Sigma[b])$ represents $y$, implies (in our setting), that $\boundary \Sigma[b]$ is a single curve intersecting each component of $\gamma\setminus b$ exactly once. We can extend $\boundary \Sigma[b]$ through the solid torus $\nbhd(K(\phi_2)) = S^3 \setminus \interior(N[b])$ so that it becomes a Seifert surface $F$ for $K(\phi_2)$ whose intersection with $\nbhd(K(\phi_2))$ is a standard collar. Since $\Sigma[b]$ is taut and represents $\pm y$, $F$ is a minimal genus Seifert surface for $K(\phi_2)$. A small isotopy in $N[b]$ ensures it is transverse to $\beta$. Since $F$ is incompressible and $N[b]$ is irreducible, there is an isotopy of $\Sigma[b]$ relative to $\boundary \Sigma$ so that $\Sigma = F \cap N$ is incompressible in $N$. Consequently, we may assume that $F$ is standard with respect to $\phi$.

Since $\beta$ is properly isotopic in $N[b]$ to an arc disjoint from $\Sigma[b]$, the surface $\Sigma[b]$ is properly isotopic in $N[b]$ to a surface $\Sigma'[b]$ disjoint from $\beta$. Hence, $x(N) \leq x(N[b])$. If $K(\phi_2)$ is the unknot, then $F$, $\Sigma[b]$, and $\Sigma'[b]$  are all discs. In which case, $\Sigma'[b]$ is an essential disc in $N$, contradicting Lemma \ref{Jaco lem}. Thus, $K(\phi_2)$ is not an unknot, as claimed.

\textbf{Case 2b:} Conclusion (SM3) holds.

In this case,
\[
-2\chi(Q) + |\boundary Q \cap \gamma| \geq 2|\boundary Q \cap b|
\]
We will encounter a contradiction.

Rewrite the inequality to take into account the different types of boundary components of $Q$:
\[
-2(1 - |\boundary_a Q| - |\boundary_{a^*} Q|) \geq (|\boundary_a Q \cap b| - |\boundary_{a} Q \cap b'|) + (|\boundary_{a^*} Q \cap b| - |\boundary_{a^*} Q \cap b'|) + |\epsilon_1 \cap b| - |\epsilon_1 \cap b'| - |\boundary_0 Q \cap \mu|.
\]
Hence, applying inequalities (I1), (I2), and (I3), and using the fact that $(a \cup a^*) \cap \mu = \nil$, we get:
\[
-2 \geq |\boundary_a Q|(|a \cap b| - |a \cap b'| - 2) + |\boundary_{a^*} Q|(|a^* \cap b| - |a^* \cap b'| - 2) + |\epsilon_1 \cap b| - |\epsilon_1 \cap b'| - 1
\]
But (Count 1), (Count 2), and (Count 3) then imply that $-2 \geq 0$, a contradiction. Thus, this case does not occur.
\end{proof}

\begin{proof}[Proof of the Theorem \ref{Main Thm}]
Suppose that $(V,\psi)$ is nontrivial and that $(W, \phi_1)$ and $(W, \phi_2)$ are inequivalent complementary 1-tangles such that $K(\phi_1)$ and $K(\phi_2)$ are unknots. By Lemma \ref{trivial}, both $(W, \phi_1)$ and $(W, \phi_2)$ are trivial. By Proposition \ref{Seifert business}, $\omega(\phi_1, \phi_2) \leq 1$. By Lemma \ref{lem: wrapping well def},  $\omega(\phi_1, \phi_2) = 1$. Our theorem then follows immediately from Corollary \ref{small wrapping}. 
\end{proof}

\section{The Krebes 1-tangle}\label{krebes section}
\begin{krebestheorem}
The Krebes 1-tangle is persistent.
\end{krebestheorem}

\begin{proof}[Proof of Theorem \ref{Krebes thm}]

Let $(V, \psi)$ be the Krebes 1-tangle, as in Figure \ref{Fig:Krebes}. Let $\phi_2$ be the arc shown on the left of Figure \ref{fig: Krebes2}, such that $K(\phi_2)$ is the Figure 8 knot. Let $b$ be the indicated longitude; it bounds a disc in $W\setminus \phi_2$. Let $\beta \subset N[b]$ be the cocore of the 2-handle.  Since $S[b]$ bounds a ball that is a regular neighborhood of $\phi'_2$, we may extend the endpoints of $\beta$ vertically through the ball to lie on $\phi_2$, creating a $\theta$-curve $\Gamma$ containing $K(\phi_2)$ as a constituent knot, having the curve $b$ as a meridian of the edge  $\beta = \Gamma \setminus K(\phi_2)$ and with $N = S^3 \setminus \nbhd(\Gamma)$. The middle and right of Figure \ref{fig: Krebes2} show two perspectives of $\Gamma$. 

Let $F'$ be the minimal genus Seifert surface for $K(\phi_2)$ given by Seifert's algorithm applied to the diagram in Figure \ref{fig: Krebes2}. As shown in Figure \ref{fig: Krebes3}, there is an isotopy of $F'$, relative to $K(\phi_2)$, so that $F'$ remains standard with respect to $\phi_2$ and so that $\beta \cap F'$ contains a single point in the interior of $F'$. (For clarity, the figure actually shows an isotopy of $\Gamma$ but reversing it gives an isotopy of $F'$.)

\begin{figure}[ht!]
\labellist
\small\hair 2pt
\pinlabel {$\phi_2$} [b] at 125 195
\pinlabel {$b$} [t] at 121 94
\pinlabel {$\phi_2$} [b] at 335 219
\pinlabel {$\beta$} [t] at 340 97
\pinlabel {$\beta$} [t] at 564 91
\endlabellist
\centering
\includegraphics[scale=0.6]{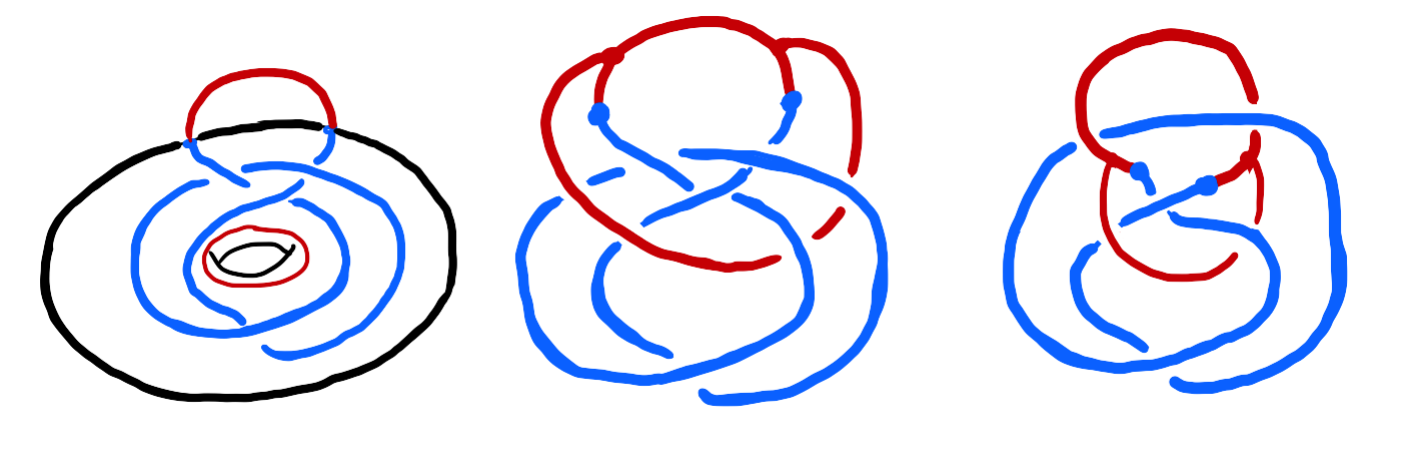}
\caption{On the left we see that the closure of the Krebes tangle using the arc $\phi_2$ is the Figure 8 knot. The preferred longitude $b$ bounds a disc in $W$ disjoint from $\phi_2$. The middle and right show two perspectives of the $\theta$-curve $\Gamma$ having $K(\phi_2)$ as a constituent knot and an extension of the arc $\beta$ as the edge $\Gamma \setminus K(\phi_2)$.  }
\label{fig: Krebes2}
\end{figure}

\begin{figure}[ht!]
\centering
\includegraphics[scale=0.4]{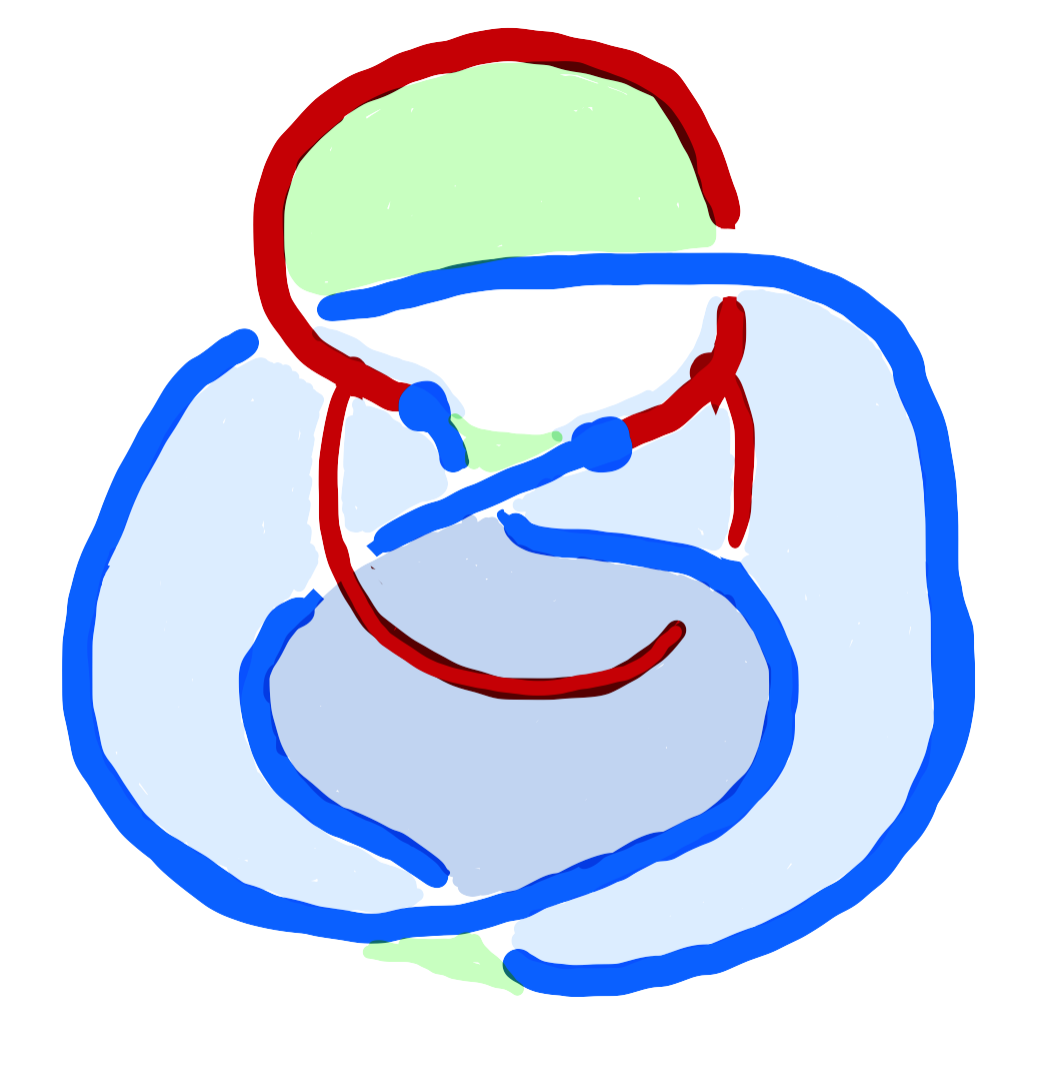}
\caption{There is a minimal genus Seifert surface $F'$ for $K(\phi_2)$ intersecting $\beta$ in a single point which is standard with respect to $\phi_2$. }
\label{fig: Krebes3}
\end{figure}

Suppose, to obtain a contradiction, that there exists  a complementary 1-tangle $(W, \phi_1)$ such that $K(\phi_1)$ is the unknot. By Lemma \ref{lem: wrapping well def}, $\omega(\phi_1, \phi_2) \geq 1$. 

Suppose, first, that $\omega(\phi_1, \phi_2) = 1$.  By Lemma \ref{wrapping 1 lem}, $\omega(\phi_2, \phi_1) = 1$. By Lemma \ref{reduction}, the 2-tangle $(V[b], \psi \cup \beta)$ admits an unknot closure. However, (inspired by \cites{SiWi, KL}), Figure \ref{fig:Krebescoloring} shows that $(V[b], \psi \cup \beta)$ admits a $\Z/3\Z$ coloring (aka a ``3-coloring'') of its strands. Note that the coloring of the endpoints is constant, so this 3-coloring can be extended to a 3-coloring of any knot into which the 2-tangle embeds. In fact, if we take an odd number of parallel copies of the arc $\beta$, we see that the resulting tangle also admits a 3-coloring. Hence, if the Krebes' 1-tangle admits an unknot closure it must be by a complementary 1-tangle $\phi_1$ with $\omega(\phi_2, \phi_1)$ even. This gives a ``low-tech'' proof of Abernathy's result \cite{Abernathy}*{Theorem 1.3} that the determinant of any knot obtained by an odd closure of the Krebes tangle has determinant a multiple of 3. Consequently, we conclude that  $\omega(\phi_1, \phi_2) \geq 2$.

\begin{figure}[ht!]
\centering
\includegraphics[scale=0.4]{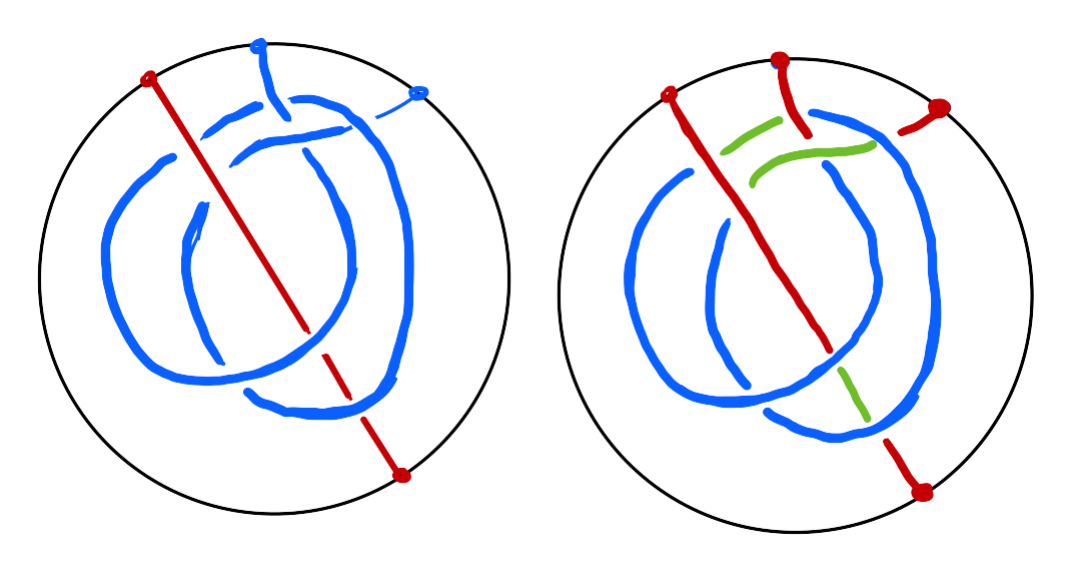}
\caption{On the left is the 2-tangle $b$-induced from the Krebes 1-tangle. On the right, we show it admits a 3-coloring with all boundary points of the same color. Hence, any knot which is the closure of the Krebes 1-tangle of wrapping index 1 from $(W, \phi_2)$ is 3-colorable. In particular, it is not the unknot. We can also see that the tangle obtained by taking an odd number of parallel copies of $\beta$ is 3-colorable, giving another proof of \cite{Abernathy}*{Theorem 1.3}.}
\label{fig:Krebescoloring}
\end{figure}

We apply Theorem \ref{Seifert business}, to conclude that one of the following occurs:
\begin{enumerate}
\item There is a complementary 1-tangle $(W, \phi'_2)$ equivalent to $(W, \phi_2)$ such that $K(\phi_2)$ is the result of a nontrivial band sum of $K(\phi_1)$ and $\kappa = \phi_1 \cup \phi'_2$. 
\item There is a minimal genus Seifert surface $F$ for $K(\phi_2)$ such that $\Sigma[b] = F \cap N[b]$ is properly isotopic in $N[b]$ to a surface disjoint from $\beta$.
\end{enumerate}

We examine these possibilities in turn.

\textbf{Case 1:} $K(\phi_2)$ is the nontrivial band sum of the unknot $K(\phi_1)$ and $\kappa = \phi_1 \cup \phi'_2$ where $\phi'_2$ is the arc of a 1-tangle equivalent to $\phi_2$.

By a theorem of Gabai \cite{Gabai} and  Scharlemann \cite{Scharlemann}*{Theorem 8.4}, 
\[
1 = g(K(\phi_2)) \geq g(\kappa) + g(K(\phi_1)) = g(\kappa).
\]
Equality holds if and only if the band is disjoint from a minimal genus Seifert surface for the split link $\kappa \cup K(\phi_1)$. However, since $K(\phi_1)$ is the unknot, this would imply that band sum is trivial, which it is not. Hence, $1 > g(\kappa) \geq 0$, so $\kappa$ is the unknot. However, this would imply it is possible to attach a band to the figure 8 knot $K(\phi_2)$ to obtain the unlink of two components. In which case, the figure 8 knot would be slice, which it is not. Thus, this case does not occur.

\textbf{Case 2:} There is a minimal genus Seifert surface $F$ for $K(\phi_2)$ such that $\Sigma[b] = F \cap N[b]$ is properly isotopic in $N[b]$ to a surface disjoint from $\beta$.

The Figure 8 knot $K(\phi_2)$ is fibered, so, ignoring $\beta$, the surface $\Sigma'[b] = F' \cap N[b]$ is isotopic rel $K(\phi_2)$ to the surface $\Sigma[b] = F \cap N[b]$. By Lemma \ref{disjoint arc}, since $F'$ is not disjoint from $\beta$, the surface $\Sigma' = F' \cap N$ admits a $d$-$\boundary$-compressing disc for some simple closed curve $d \subset \boundary N[b]$ intersecting $\boundary \Sigma'$ minimally up to isotopy.  Let $D$ be the $\boundary$-compressing disc. Let $\delta' = \boundary D \cap \Sigma'$ and $\delta = \boundary D \cap \boundary N[b]$. Since $F'$ is an incompressible Seifert surface, the circle $\boundary D = \delta \cup \delta'$ bounds a disc in $\Sigma'[b] \cup \boundary N[b]$ that intersects $\boundary \Sigma'[b]$ in a single arc. That arc divides the disc into two subdiscs $D' \subset \Sigma'[b]$ and $D'' \subset \boundary N[b]$. Since $|F' \cap \beta| = 1$, the disc $D'$ contains that single point of intersection. Hence, $D''$ contains a unique endpoint of $\boundary \beta$. Thus, the sphere $P = D \cup D' \cup D''$ bounds a 3-ball $B \subset N[b]$ intersecting $\beta$ in a single arc $\beta'$. Since $\beta$ is contained in an unknot in $S^3$, the arc $\beta'$ is isotopic relative to $\boundary \beta'$ into $P$. We may arrange that the isotopy takes $\beta'$ to an arc intersecting $\boundary \Sigma'[b] \cap (D' \cup D'')$ exactly once. 
We now show this is impossible.

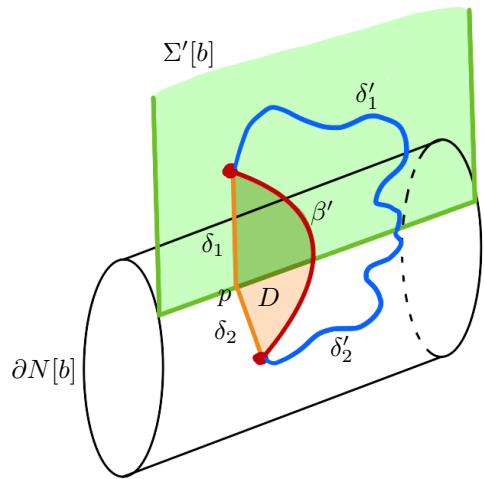
\begin{figure}[ht!]
\labellist
\small\hair 2pt
\pinlabel {$D$} at 134 132
\pinlabel {$\beta'$} [l] at 157 185
\pinlabel{$\delta_1$} [r] at 110 165
\pinlabel{$\delta_2$} [r] at 118 110
\pinlabel{$\delta'_1$} [b] at 196 247
\pinlabel{$\delta'_2$} [t] at 180 111
\pinlabel{$\boundary N[b]$} [r] at 17 85
\pinlabel{$p$} [tr] at 113 139
\pinlabel{$\Sigma'[b]$} [b] at 84 273
\endlabellist
\centering
\includegraphics[scale=0.6]{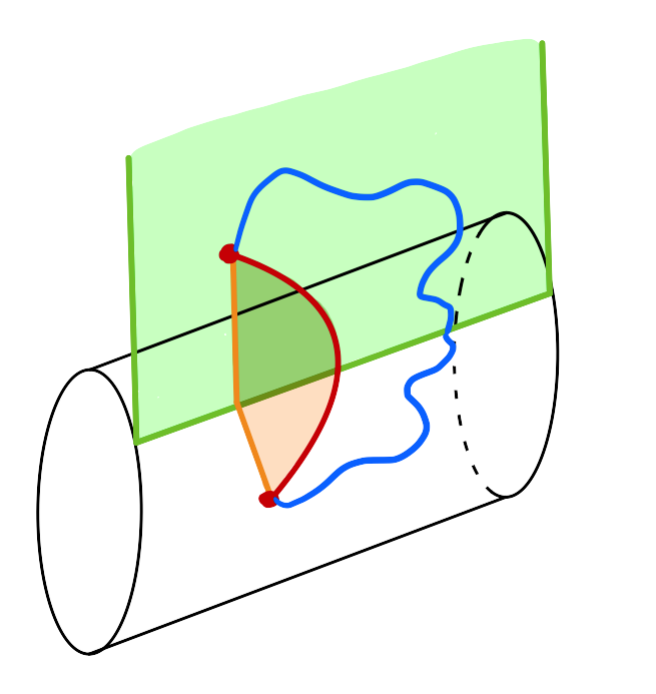}
\caption{The disc $D$ giving a parallelism of $\beta'$ into $\Sigma'[b] \cup \boundary N[b]$.}
\label{fig: boundarycomp}
\end{figure}

Recycling notation, let $D \subset N[b]$ be an embedded disc with $\boundary D$ the union of three arcs $\beta'$, $\delta_1$, and $\delta_2$. See Figure \ref{fig: boundarycomp}. The arc $\beta'$ is a component of $\beta \setminus F'$. The arc $\delta_1 \subset F'$ and the arc $\delta_2 \subset \boundary N[b]$. Let $p = \boundary D \cap \boundary \Sigma'[b]$ and orient $\boundary D$ so that $p$ is the initial endpoint of $\delta_1$. 

\emph{A priori}, we do not have much information about the paths $\delta_1$ and $\delta_2$, so we aim to replace them with other arcs. Let $\delta'_1 \subset \Sigma'[b]$ be any path joining a point $p' \in \boundary \Sigma'[b]$ to $\beta \cap F$. Let $\delta'_2 \subset \boundary N[b]$ be any path in $\boundary N[b]$ with interior disjoint from $\boundary \Sigma'[b]$ and joining the initial endpoint of $\delta_2$ to $p'$. Let $\lambda$ be a path in $\boundary \Sigma'[b]$ from $p$ to $p'$. For an oriented path $\alpha$, let $\overline{\alpha}$ denote its reverse and let $[\alpha]$ denote its based homotopy class in $\pi_1(N[b],p')$. We concatenate paths from left to right. Then,
\[
[\delta'_1 ~\beta'~ \delta'_2] = [\delta'_1 ~ \overline{\delta_1} ~ \lambda ] [\overline{\lambda} ~ \delta_1 ~ \beta' ~ \delta_2 ~ \lambda][\overline{\lambda}~ \overline{\delta_2} ~ \delta'_2] = [\delta'_1~  \overline{\delta_1} ~ \lambda ][\overline{\lambda}~\overline{\delta_2} ~ \delta'_2]
\]
The class $[\delta'_1 ~ \overline{\delta_1}~  \lambda ]$ lies in the image $H$ of $\pi_1(\Sigma'[b], p')$ induced by the inclusion of $\Sigma'[b]$ into $N[b]$.  The class $[\overline{\lambda}~ \overline{\delta_2}~  \delta'_2]$ lies in the subgroup of $\pi_1(N[b], p')$ generated by the preferred longitude of $\boundary N[b]$. Hence, $[\delta'_1 ~ \beta'~  \delta'_2]$ lies in $H$.

As shown in Figure \ref{fig: meridians}, for either choice of $\beta'$, we may choose the arcs $\delta'_1$ and $\delta'_2$, so that $\delta'_1 \cup \beta' \cup \delta'_2$ is a meridian of $K(\phi_2)$. But, for any knot, it is the case that a meridian is never freely homotopic into the image of the fundamental group of a Seifert surface induced by the inclusion of the Seifert surface into the knot complement. Thus, there is no unknot closure of the Krebes tangle.
\end{proof}

\begin{figure}[ht!]
\labellist
\small\hair 2pt
\pinlabel {$\delta'_1 \cup \delta'_2$} [l] at 296 246
\pinlabel {$\delta'_1 \cup \delta'_2$} [l] at 566 49
\endlabellist
\centering
\includegraphics[scale=0.5]{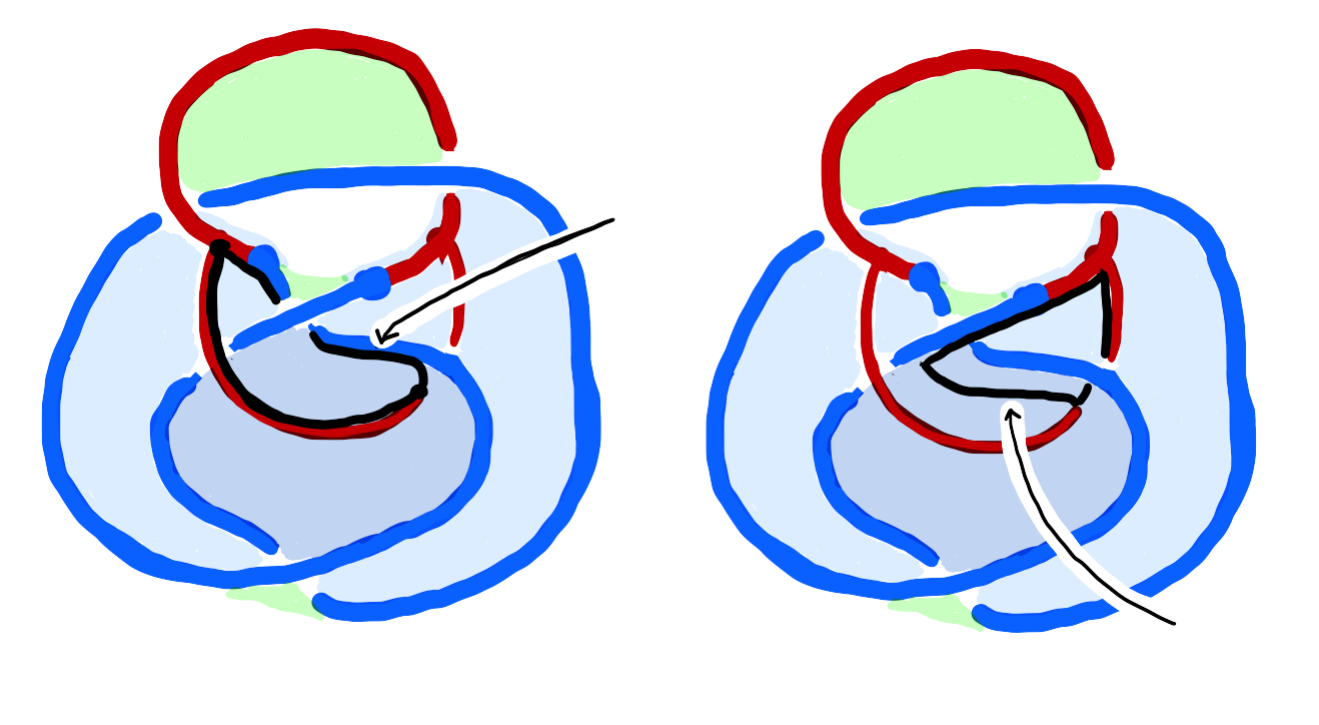}
\caption{For either choice of $\beta'$ we can choose the arcs $\delta'_1$ and $\delta'_2$ so that $\delta'_1 \cup \beta' \delta'_2$ is a meridian of the knot.}
\label{fig: meridians}
\end{figure}

\section{Acknowledgements}
I am grateful to Colin Adams, Ken Baker, Luis Celso Chan, and Nathan Dunfield for helpful conversations. In particular, Ken Baker drew my attention to the Krebes problem. This work was partially supported by NSF Grant DMS-2104022 and a Colby Research Grant.

\end{document}